\documentclass[12pt,leqno]{article}
\usepackage{amsfonts}
\usepackage{amsmath, amsthm, amsfonts, amssymb, color}
\usepackage{mathrsfs}
\usepackage{url}
\usepackage{xcolor}
\newtheorem{thm}{Theorem}[section]

\newtheorem{lem}[thm]{Lemma}

\newtheorem{rem}[thm]{Remark}
\theoremstyle{definition}

\newcommand{\scr}[1]{\mathscr #1}
\definecolor{wco}{rgb}{0.5,0.2,0.3}

\numberwithin{equation}{section} \theoremstyle{remark}

\newcommand{\ua}{\uparrow}

\title{{\bf  Wasserstein Convergence for Empirical Measures of Subordinated Dirichlet Diffusions on Riemannian Manifolds }
}

\author{Huaiqian Li\footnote{Email: {\color{blue}huaiqianlee@gmail.com}}
\quad Bingyao Wu\footnote{Email: {\color{blue}bingyaowu@163.com}}
  \vspace{2mm}
\\
{\footnotesize Center for Applied Mathematics, Tianjin University, Tianjin 300072, P. R. China}
}

\date{}

\begin{document}
\allowdisplaybreaks
\def\R{\mathbb R}  \def\ff{\frac} \def\ss{\sqrt} \def\B{\mathbf
B}\def\TO{\mathbb T}
\def\I{\mathbb I_{\pp M}}\def\p<{\preceq}
\def\N{\mathbb N} \def\kk{\kappa} \def\m{{\bf m}}
\def\ee{\varepsilon}\def\ddd{D^*}
\def\dd{\delta} \def\DD{\Delta} \def\vv{\varepsilon} \def\rr{\rho}
\def\<{\langle} \def\>{\rangle} \def\GG{\Gamma} \def\gg{\gamma}
  \def\nn{\nabla} \def\pp{\partial} \def\E{\mathbb E}
\def\d{\text{\rm{d}}} \def\bb{\beta} \def\aa{\alpha} \def\D{\scr D}
  \def\si{\sigma} \def\ess{\text{\rm{ess}}}
\def\beg{\begin} \def\beq{\begin{equation}}  \def\F{\scr F}
\def\Ric{{\rm Ric}} \def\Hess{\text{\rm{Hess}}}
\def\e{\text{\rm{e}}} \def\ua{\underline a} \def\OO{\Omega}  \def\oo{\omega}
 \def\ttt{\tilde}\def\tt{\theta}
\def\cut{\text{\rm{cut}}} \def\P{\mathbb P} \def\ifn{I_n(f^{\bigotimes n})}
\def\C{\scr C}      \def\aaa{\mathbf{r}}     \def\r{r}
\def\gap{\text{\rm{gap}}} \def\prr{\pi_{{\bf m},\varrho}}  \def\r{\mathbf r}
\def\Z{\mathbb Z} \def\vrr{\varrho} \def\ll{\lambda}
\def\L{\scr L}\def\Tt{\tt} \def\TT{\tt}\def\II{\mathbb I}
\def\i{{\rm in}}\def\Sect{{\rm Sect}}  \def\H{\mathbb H}
\def\M{\scr M}\def\Q{\mathbb Q} \def\texto{\text{o}} \def\LL{\Lambda}
\def\Rank{{\rm Rank}} \def\B{\scr B} \def\i{{\rm i}} \def\HR{\hat{\R}^d}
\def\to{\rightarrow}\def\l{\ell}\def\iint{\int}
\def\EE{\scr E}\def\Cut{{\rm Cut}}\def\W{\mathbb W}
\def\A{\scr A} \def\Lip{{\rm Lip}}\def\S{\mathbb S}
\def\BB{\mathbf B}\def\Ent{{\rm Ent}} \def\i{{\rm i}}\def\itparallel{{\it\parallel}}
\def\g{{\mathbf g}}\def\Sect{{\mathcal Sec}}\def\T{\mathcal T}\def\V{{\bf V}}
\def\PP{{\bf P}}\def\HL{{\bf L}}\def\Id{{\rm Id}}\def\f{{\bf f}}\def\cut{{\rm cut}}
\def\Ss{\mathbb S}
\def\BL{\scr A}\def\Pp{\mathbb P}\def\Pp{\mathbb P} \def\Ee{\mathbb E}
\def\hp{\hat{\phi}}\def\vv{\varepsilon}\def\ww{\wedge}
\maketitle

\begin{abstract}
We investigate long-time behaviors of empirical measures associated with subordinated Dirichlet diffusion processes on a compact Riemannian manifold $M$ with boundary $\partial M$ to some reference measure, under the quadratic Wasserstein distance. For any initial distribution not concentrated on $\partial M$, we obtain the rate of convergence and even the precise limit for the conditional expectation of the quadratic Wasserstein distance conditioned on the process killed upon exiting $M\setminus\partial M$. In particular, the results coincide with the recent ones proved by F.-Y. Wang in \cite{eW2} for Dirichlet diffusion processes.
 \end{abstract} \noindent
{\bf MSC 2020:}  primary  60D05, 58J65; secondary 60J60, 60J76\\
\noindent
{\bf Keywords:}  Empirical measure, subordinated diffusion process, Riemannian manifold, Wasserstein distance, eigenvalue
 \vskip 2cm

\section{Introduction and main results}
Let $M$ be a $d$-dimensional connected compact Riemannian manifold with smooth  boundary $\pp M$ and $\scr{P}$ be the class of all Borel probability measures on $M$.  Let $V\in C^2(M)$ such that $\mu(\d x)=\e^{V(x)} \d x$ and $\mu\in\scr{P}$, where $\d x$ stands for the Riemannian volume measure on $M$. Let $(X_t)_{t\geq0}$ be the diffusion process generated by $\mathcal{L}:=\DD+\nn V$ with the hitting time
$$\tau:=\inf\{t\ge 0:X_t\in\pp M\},$$
where $\Delta$ and $\nabla$  are the Laplace--Beltrami operator and the gradient operator on $M$, respectively.

Let $x\in M$ and $\nu\in\scr{P}$. We use the usual notation $\P^\nu(\cdot)=\int_M \P^x(\cdot)\,\nu(\d x)$,  where $\P^x$ stands for the law of the corresponding process with initial point $x$.  The expectations $\E^x$ and $\E^\nu$ are associated with $\P^x$ and $\P^\nu$, respectively. As usual, for every $p\in[1,\infty]$, we use $L^p(\nu):=L^p(M,\nu)$ to denote the $L^p$ function space with norm $\|\cdot\|_{L^p(\nu)}$.  For every $f\in L^1(\nu)$, let $\nu(f)$ be the shorthand notation for $\int_M f\,\d\nu$.

For every $q\in(0,\infty)$, the Wasserstein (or Kantorovich) distance with power $q$, induced by the Riemannian distance $\rr$ on $M$, is defined as
$$\W_q(\mu_1,\mu_2)= \inf_{\pi\in \C(\mu_1,\mu_2)} \bigg(\int_{M\times M} \rr(x,y)^q \pi(\d x,\d y) \bigg)^{\frac{1}{q\vee1}}\in[0,\infty],\quad \mu_1,\mu_2\in \scr P,$$
where $q\vee1:=\max\{q,1\}$ and $\C(\mu_1,\mu_2)$ is the class of all Borel probability measures on the product space $M\times M$ with
marginal distributions $\mu_1$ and $\mu_2$, respectively. In the sequel, we will mainly focus on $\W_2$, which is also called quadratic Wasserstein distance. For instance, see \cite[Chapter 5]{ChenMF2004}, \cite{Villani2008} and \cite{Santam2015} for a comprehensive study on the Wasserstein distance and its connection to the optimal transport theory.

Let $\mathbb{N}=\{1,2,\cdots\}$ and $\mathbb{N}_0=\mathbb{N}\cup\{0\}$.  Let $\phi_m,\lambda_m$, $m\in\mathbb{N}_0$, be Dirichlet eigenfunctions and Dirichlet eigenvalues of the operator $-\mathcal{L}$ in $L^2(\mu)$ respectively (see Section 2 below for more details). Set
$$\mu_0:=\phi_0^2\mu.$$
It is clear that $\mu_0\in\scr{P}$. Let $\delta_\cdot$ stands for the Dirac measure. Consider the family of empirical measures associated with the diffusion process $(X_t)_{t\geq0}$, i.e.,
$$\mu_t:=\frac{1}{t}\int_0^t \delta_{X_s}\,\d s,\quad t>0,$$
which are random probability measures. In a recent paper \cite{eW2}, F.-Y. Wang studied the rate of convergence of $\E^\nu[\W_2(\mu_t,\mu_0)^2|t<\tau]$ as $t$ tends to $\infty$, where the initial distribution $\nu$ belongs to $\scr P_0$. Here and in what follows,
$$\scr P_0:=\{\nu\in\scr P:\ \nu(\mathring{M})>0\},$$
and $\mathring{M}:=M\setminus\partial M$, the interior of $M$.

The main purpose of the present work is to extend the main results in \cite{eW2} to a large class of Markov processes subordinated to $(X_t)_{t\geq0}$.

In order to introduce our main results, we should begin by recalling some basics that are of importance in the sequel. A continuous function $B:[0,\infty)\rightarrow[0,\infty)$ is called a Bernstein function if, $B$ is infinitely differentiable in $(0,\infty)$ and for every $n\in\mathbb N$,
 $$ (-1)^{n} \ff{\d^n}{\d r^n}B(r)\leq0,\quad  r>0.$$
The following class of  Bernstein functions is of importance, i.e.,
\beg{align*}
\mathbf{B}:= \big\{B: B\text{\ is\ a\ Bernstein\ function\ with } B(0)=0, B'(0)>0\big\}.
\end{align*}
It is known that, for every $B\in {\bf  B}$, there exists a unique subordinator $(S_t^B)_{t\geq0}$, which means an increasing process $S^B: [0,\infty)\rightarrow [0,\infty)$ with stationary, independent increments and $S_0^B=0$ almost surely such that $B$ is the Laplace exponent of $(S_t^B)_{t\geq0}$, i.e.,
\begin{equation}\label{LT}
\E \e^{-\ll S_t^B}= \e^{-t B(\ll)},\ \ t,\ll \ge 0.
\end{equation}
In particular, when $B(\ll)=\ll^\aa$, $\aa\in(0,1)$, the  corresponding process $(S_t^B)_{t\ge 0}$ is also called $\aa$-stable subordinator.
For any $\aa\in (0,1]$, let
$$\BB^\aa:=\Big\{B\in \BB:\ \liminf_{\ll\to\infty}  \ll^{-\aa} B(\ll)>0\Big\},\quad \BB_\aa:= \Big\{B\in \BB:\ \limsup_{\ll\to\infty} \ll^{-\aa} B(\ll)<\infty\Big\}.$$
Typical examples belong to the classes $\BB^\aa$ and $\BB_\aa$ are $B(\ll)=\ll$ and $B(\ll)=\ll^\alpha$ with $\alpha\in(0,1)$. The former goes back to the situation of \cite{eW2}, and the later corresponds to the setting of the classical $2\alpha$-stable process on $M$ killed upon exiting $\mathring{M}$ (see e.g. \cite{SV2003}). In addition, there are many other types of Bernstein functions belong to the class $\BB^\aa$ or $\BB_\aa$ (see e.g. \cite[Chapter 16]{SSV2012}).

Let $B\in {\bf B}$. As defined above,  let $(S^B_t)_{t\ge 0} $ be the unique subordinator independent of $(X_t)_{t\geq0}$. The Markov process on $M$ generated by $-B(-\mathcal{L})$ is denote by $(X_t^B)_{t\geq0}$, which can be regarded as the time-changed process of  $(X_t)_{t\geq0}$, i.e.,
 $$X_t^B= X_{S^B_t\wedge\tau},\quad t\ge 0.$$
We call $(X_t^B)_{t\geq0}$ the Dirichlet diffusion process subordinated to $(X_t)_{t\geq0}$ or $B$-subordinated Dirichlet diffusion process. For more details on Bernstein functions and subordinated processes, refer to \cite{SSV2012,Bertoin97,BSW} for instance.

Let
$$\si_\tau^B:=\inf\{t\ge0:\ S_t^B>\tau\},$$
which can be regarded as the hitting time of the $B$-subordinated Dirichlet diffusion process $(X_t^B)_{t\geq0}$ at the boundary $\partial M$. We consider the family of empirical measures associated with $(X_t^B)_{t\geq0}$, defined by
 $$\mu_t^B=\ff 1 t \int_0^t \dd_{X_{s}^B}\,\d s,\quad t>0.$$
Recall that $\mu_0=\phi_0^2\mu$. As announced earlier, we will investigate the rate of convergence of $\E^\nu[\W_2(\mu_t^B,\mu_0)^2|t<\si_\tau^B]$ as $t\to\infty$ for every initial distribution $\nu\in\scr{P}_0$. Note that the reason to consider $\nu$ from $\scr{P}_0$ instead of $\scr{P}$ is to avoid $\Pp^\nu(\si_\tau^B>t)=0$.

Our first main result is stated in the next theorem.
\beg{thm}[Upper  bounds] \label{T1.2} Let $\aa\in (0,1]$
and $B\in \BB^\aa$. Then for any $\nu\in\scr{P}_0$, the following assertions hold.
\beg{enumerate}
 \item[$(1)$] If $d< 2(\alpha+1)$, then there exists a constant $c>0$, such that
\begin{equation}\label{0T1.2}
 \limsup_{t\to \infty}  \sup_{T\ge t}   \Big\{t \E^\nu [\W_2(\mu^B_{t},\mu_0)^2|T<\si_\tau^B]\Big\}\le
 c\sum_{m=1}^\infty\ff 2 {(\ll_m-\ll_0) [B(\ll_m)-B(\ll_0)]}<\infty.
\end{equation}
If $\aa\in(\ff 1 2,1]$ and $d<6\aa-2$,  then \eqref{0T1.2} holds with $c=1$.

\item[$(2)$] If $d>2(1+\aa)$, then
\begin{equation}\label{2T1.2} \limsup_{t\to \infty}  \sup_{T\ge t}    \Big\{t^{\ff{2}{d-2\aa}} \E^\nu [\W_2(\mu^B_{t},\mu_0)^2|T<\si_\tau^B]\Big\}<\infty.
\end{equation}
\item[$(3)$] If $d=2(1+\aa)$, i.e., either $(\aa,d)=(1,4)$ or $(\aa,d)=(\ff 1 2,3)$, then
\begin{equation}\label{3T1.2} \limsup_{t\to \infty}  \sup_{T\ge t}    \Big\{ \ff t{\log t}  \E^\nu [\W_2(\mu^B_{t},\mu_0)^2|T<\si_\tau^B]\Big\}<\infty.
\end{equation}
\end{enumerate}
   \end{thm}
\begin{rem}
(1) If $\aa\in (0,1]$, $B\in\mathbf{B}_\alpha$ and $d\geq2(1+\alpha)$, then the first inequality of \eqref{0T1.2} trivially holds. Indeed, by \eqref{equ-B}, \eqref{BR} and the fact that $B(\ll)\leq k_1 \ll^\alpha+k_2$, $s\geq0$, for some constant $k_1>0$ and $k_2\geq0$ (see e.g. \cite[Section 4.2]{eWW}), we have
$$\sum_{m=1}^\infty\ff 1 {(\ll_m-\ll_0) [B(\ll_m)-B(\ll_0)]}\geq C\sum_{m=1}^\infty\frac{1}{m^{2(1+\alpha)/d}}=\infty,$$
where $C$ is a positive constant.

(2) Note that, in the second assertion of Theorem \ref{T1.2}(1), there is no need to consider the case when $\alpha\in(0,\frac{1}{2}]$ since $6\alpha-2\leq1$.
\end{rem}

To present our second main result, we first recall that the boundary $\partial M$ is convex means that the second fundamental form of $M$ is nonnegative definite.
\beg{thm}[Lower bounds]\label{T1.1}
Let $\nu\in\scr{P}_0$. The following assertions hold.
\beg{enumerate} \item[$(1)$] Let $B\in\mathbf{B}$.
There exists a constant $c\in (0,1]$ such that
\begin{equation}\label{1T1.1}
\liminf_{t\to \infty}   \Big\{t \inf_{T\ge t}\E^\nu [\W_2(\mu^B_{t},\mu_0)^2|T<\si_\tau^B]\Big\}
\ge   c  \sum_{m=1}^\infty\ff 2 {(\ll_m-\ll_0) [B(\ll_m)-B(\ll_0)]}.
\end{equation}
Moreover, if $\pp M$ is convex, then \eqref{1T1.1} holds with $c=1$.

\item[$(2)$] Let $B\in \BB_\aa$   for some  $\aa\in (0,1]$. If $d>2(1+\aa)$, then  for any $p>0$,
\begin{equation}\label{2T1.1}  \liminf_{t\to \infty}    \Big\{t^{\ff 2{d-2\aa}}\inf_{T\ge t}\big( \E^\nu [\W_p(\mu^B_{t},\mu_0)|T<\si_\tau^B]\big)^{\ff 2{p\land 1}} \Big\}>0.\end{equation}
\end{enumerate}
   \end{thm}
\begin{rem}
From the proof of Theorem \ref{T1.1} in Section 4, we should point out that the assumption $B\in \mathbb{B}$ in \cite[Theorem 1.1(1)]{eWW} is not necessary and it can be weakened by $B\in\mathbf{B}$, where
$$\mathbb{B}:= \left\{B\in {\bf B}:\  \int_1^\infty s^{\ff d 2 -1}\e^{-r B(s)}\,\d s<\infty\mbox{ for all }r>0\right\}.$$
Obviously, $\mathbb{B}\subset\mathbf{B}$; however, the converse is not true. For instance,  letting $\aa\in(0,1)$ and setting $B(\ll):=1-(1+\ll)^{\aa-1}$ for any $\lambda\geq0$, we can easily verify that $B$ belongs to $\BB$ but it is not in $\mathbb{B}$.
\end{rem}

As an immediate consequence of Theorems \ref{T1.1} and \ref{T1.2}, we have the following corollary.
\beg{cor}\label{C1.3} Let $\nu\in\scr{P}_0$.
\beg{enumerate} \item[$(1)$] Suppose that $\partial M$ is convex. If $\alpha\in(\frac{1}{2},1]$, $B\in \BB^\aa$ and   $d<6\alpha-2$, then
\begin{equation*}\begin{split}
\lim_{t\to \infty}   \Big\{t \inf_{T\ge t}\E^\nu [\W_2(\mu^B_{t},\mu_0)^2|T<\si_\tau^B]\Big\}
&=\lim_{t\to \infty}   \Big\{t \sup_{T\ge t}\E^\nu [\W_2(\mu^B_{t},\mu_0)^2|T<\si_\tau^B]\Big\}\\
&=\sum_{m=1}^\infty\ff 2 {(\ll_m-\ll_0) [B(\ll_m)-B(\ll_0)]}<\infty.
\end{split}\end{equation*}

\item[$(2)$] If $\alpha\in(0,1]$, $B\in \BB^\aa\cap \BB_\aa$ and $d>2(\alpha+1)$, then for every $p\in(0,\alpha)$, there exist constants $c\geq c_p>0$ such that
\begin{equation*}\begin{split}
c_pt^{-\frac{2}{d-2\alpha}}&\leq \inf_{T\ge t}\big(\E^\nu [\W_p(\mu^B_{t},\mu_0)|T<\si_\tau^B]\big)^{2/p}
\leq \inf_{T\ge t} \E^\nu [\W_2(\mu^B_{t},\mu_0)^2|T<\si_\tau^B] \\
&\leq \sup_{T\ge t} \E^\nu [\W_2(\mu^B_{t},\mu_0)^2|T<\si_\tau^B]\leq c t^{-\frac{2}{d-2\alpha}},
\end{split}\end{equation*}
for all big enough $t>0$.

\item[$(3)$] If $\alpha\in(0,1]$, $B\in \BB^\aa$ and $d=2(\alpha+1)$, then there exists a constant $c>0$ such that
$$\sup_{T\ge t} \E^\nu [\W_2(\mu^B_{t},\mu_0)^2|T<\si_\tau^B]\leq c t^{-1}\log t,$$
for all big enough $t>0$.
 \end{enumerate}
\end{cor}

Some further remarks on the results above are in order.
\begin{rem}
(i) In particular, if $B(\ll)=\ll$ for every $\ll\geq0$, then results presented in Theorems \ref{T1.2} and \ref{T1.1} turn out to be exactly the ones in \cite[Theorem 1.1]{eW2}.

(ii) In the same setting of the present paper but with a distinct target, we studied recently the rate of convergence of $\W_2(\mu_t^{B,\nu},\mu_0)^2$ in \cite{LiWu2204}, where $(\mu_t^{B,\nu})_{t>0}$ is a family of conditioned empirical measures defined by
$$\mu_t^{B,\nu}=\E^\nu\left(\left.\ff 1 t\int_0^t \dd_{X_s^B}\,\d s \right|\si_\tau^B>t\right).$$
From \cite[Theorems 1.1 and 1.3]{LiWu2204},  as $t\rightarrow\infty$, we see that $\W_2(\mu_t^{B,\nu},\mu_0)^2$ behaves of order $t^{-2}$, which is faster than the decay order of $\E^\nu [\W_2(\mu^B_{t},\mu_0)^2|t<\si_\tau^B]$ obtained in Theorems \ref{T1.2} and \ref{T1.1} above. The order of decay in the later coincides with the one of $\E^\nu[\W_2(\mu_t^B,\mu)^2]$ proved recently in \cite[Theorems 1.1 and 1.2]{eWW} (see also \cite[Theorems 2.3 and 2.7]{LiWu2201} for extensions on noncompact Riemannian manifolds), where $(\mu_t^B)_{t>0}$ is a family of empirical measures associated with the $B$-subordinated (reflected) diffusion process $(Y_t^B)_{t\geq0}$ on compact Riemannian manifolds (with boundary), i.e.,
$$\mu_t^B:=\ff 1 t \int_0^t\delta_{Y_s^B}\,\d s,\quad t>0,$$
and $\mu$ is the invariant probability measure  of $(Y_t^B)_{t\geq0}$.

(iii) In the critical dimension case, i.e., $d=2(\alpha+1)$, as $t\rightarrow\infty$, the lower bound estimate on the rate of convergence of $\E^\nu [\W_2(\mu^B_{t},\mu_0)^2|t<\si_\tau^B]$ with the matching order $t^{-1}\log t$ is a challenging problem. The problem is even open in the particular $B(\ll)=\ll$ case, i.e., in the setting of diffusion processes generated by $\mathcal{L}$ on closed $4$-dimensional Riemannian manifolds; however, see \cite[Theorem 1.3]{eWZ} for a partial answer on the torus $\mathbb{T}^4$.
\end{rem}

In the literature, besides the works cited above, let us mention the very recent papers \cite{eW1} and \cite{eW3} on quantifying the long time asymptotic behavior of empirical measures associated with diffusion processes respectively on compact and on noncompact Riemannian manifolds under the Wasserstein distance. See also \cite{eW4} for further investigations on this subject for SPDEs. We should mention that the study on asymptotic behaviours of empirical measures associated with i.i.d. random variables under Wasserstein distances, especially on estimating the rate of convergence, has been attracted a lot of attentions over years; see \cite{FG2015,BLG2014,DSS2013,GL2000} to name a few. See also the recent breakthrough \cite{eAST} on the optimal matching problem and related papers \cite{Zhu,Bor2021,BoLe2021,LedZhu2021,St2021,BL2019,Led2017} for instance. Chapter 4 of the recent book \cite{Tal2021} is also recommended.

The remainder of the paper is organized as follows. In Section 2, we gather some notations and known results on the Dirichlet eigenvalue, the Dirichlet eigenfunction, the Dirichlet diffusion semigroup and the Dirichlet heat kernel, as well as the subordinated case. Sections 3 and 4 serve to provide proofs of Theorem \ref{T1.2} and Theorem \ref{T1.1}, respectively. An appendix is included to present some detailed calculations for the reader's convenience. Let us simply mention that, even if the presentation of our work has some similarities to \cite{eW2}, new ideas are necessary in order to overcome problems in the present non-local setting.

\section{Preparations}
The results presented in this section are mainly borrowed from \cite[Sections 1 and 2]{eW2} and \cite[Section 2]{LiWu2204}. For more details, one should consult \cite{Wang2014,Ouhabaz05,Davies89,Chavel} for instance.

It is well known that the diffusion operator $-\mathcal{L}$ defined on $M$ has only discrete spectrum, which consists of an increasing sequence $(\lambda_m)_{m\in\mathbb{N}_0}$ of nonnegative eigenvalues of $-\mathcal{L}$, counting multiplicities, such that $\lambda_m\rightarrow\infty$ as $m$ tends to $\infty$. Let $(\phi_m)_{m\in\mathbb{N}_0}$ be the complete orthonormal system in $L^2(\mu)$ such that, for each $m\in\mathbb{N}_0$, $\phi_m$ is an eigenfunction corresponding to $\lambda_m$ satisfying the Dirichlet boundary condition, i.e.,
$$-\mathcal{L}\phi_m=\lambda_m\phi_m,\quad\quad \phi_m|_{\partial M}=0.$$
For each $m\in\mathbb{N}_0$, we call $\lambda_m$ and $\phi_m$ the Dirichlet eigenvalue and the Dirichlet eigenfunction of $-\mathcal{L}$, respectively. Without loss of generality, we assume that $\phi_0>0$ in $\mathring{M}$.  It is also well known that $\lambda_0>0$ and
\begin{equation}\label{EIG}
\|\phi_m\|_\infty\leq \alpha_0\sqrt{m},\quad \alpha_0^{-1}m^{\frac{2}{d}}\leq\lambda_m-\lambda_0\leq\alpha_0 m^{\frac{2}{d}},\quad\quad m\in\mathbb{N},
\end{equation}
for some constant $\alpha_0>1$; moreover,
\begin{equation}\label{PHI}
\|\phi_0^{-1}\|_{L^p(\mu_0)}<\infty,\quad p\in[1,3).
\end{equation}

Let $(p_t^D)_{t>0}$ and $(P_t^D)_{t\geq0}$ be the Dirichlet heat kernel and the Dirichlet diffusion semigroup corresponding to $\mathcal{L}$, respectively.
It is well known that $p_t^D$ admits the following expansion, i.e.,
\begin{equation}\label{DHK}
p_t^D(x,y)=\sum_{m=0}^\infty \e^{-\lambda_m t}\phi_m(x)\phi_m(y),\quad t>0,\,x,y\in M,
\end{equation}
and hence, the Dirichlet diffusion semigroup can be expressed as
\begin{equation}\begin{split}\label{DSG}
P_t^D f(x)&:=\E^x[f(X_t)1_{\{t<\tau\}}]=\int_M p_t^D(x,y)f(y)\,\mu(\d y)\\
&=\sum_{m=0}^\infty \e^{-\lambda_m t}\mu(\phi_m f)\phi_m(x),\quad t>0,f\in L^2(\mu).
\end{split}\end{equation}
Moreover, there exists a constant $c>0$ such that
\begin{equation}\begin{split}\label{DPQ}
\|P_t^D\|_{L^p(\mu)\to L^q(\mu)}&:=\sup_{\mu(|f|^p)\leq 1}\|P_t^D f\|_{L^q(\mu)}\\
&\leq c\e^{-\lambda_0 t}(1\wedge t)^{-\frac{d(q-p)}{2pq}},\quad t>0,\,\infty\geq q\geq p\geq 1.
\end{split}\end{equation}

Consider the Doob transform of $P_t^D$ defined by
\begin{equation}\label{R}
P_t^0 f:=\e^{\lambda_0 t}\phi_0^{-1}P_t^D(f\phi_0),\quad f\in L^2(\mu_0),\,t\geq0.
\end{equation}
It is well known that $(P_t^0)_{t\geq0}$ is a conservative (i.e., $P_t^01=1$ for every $t\geq0$) and symmetric semigroup in $L^2(\mu_0)$; in particular, $\mu_0$ is the invariant measure of $(P_t^0)_{t\geq0}$. The generator $\mathcal{L}_0$ of $(P_t^0)_{t\geq0}$, which is non-positive self-adjoint in $L^2(\mu_0)$, is connected with $\mathcal{L}$ by
 $$\mathcal{L}_0=\mathcal{L}+2\nabla \log\phi_0.$$
For every $m\in\mathbb{N}_0$, replacing $f$ in \eqref{R} by $\phi_0^{-1}\phi_m$, since $P_t^D\phi_m=\e^{-\ll_m t}\phi_m$, we clearly have
\begin{equation}\begin{split}\label{EIG0}
&P_t^0(\phi_m\phi_0^{-1})=\e^{-(\lambda_m-\lambda_0)t}\phi_m\phi_0^{-1},\quad t\geq 0,\\
&\mathcal{L}_0(\phi_m\phi_0^{-1})=-(\lambda_m-\lambda_0)\phi_m\phi_0^{-1}.
\end{split}\end{equation}
Then $\{\phi_0^{-1}\phi_m\}_{m\in\mathbb{N}_0}$ is an eigenbasis of $-\mathcal{L}_0$ in $L^2(\mu_0)$. Hence, by \eqref{DSG} and \eqref{R},
\begin{equation}\label{SG0}
P_t^0 f=\sum_{m=0}^\infty\mu_0(f\phi_m\phi_0^{-1})\e^{-(\lambda_m-\lambda_0)t}\phi_m\phi_0^{-1},\quad f\in L^2(\mu_0),\,t\geq0,
\end{equation}
and the heat kernel $(p_t^0)_{t>0}$ of $(P_t^0)_{t\geq0}$ w.r.t. $\mu_0$, can be expressed as
\begin{equation}\label{HK0}
p_t^0(x,y)=\sum_{m=0}^\infty(\phi_m\phi_0^{-1})(x)(\phi_m\phi_0^{-1})(y)\e^{-(\lambda_m-\lambda_0)t},\quad x,y\in M,\,t>0.
\end{equation}

By the intrinsic ultracontractivity, we can find a constant $\alpha_1\geq 1$ such that
\begin{equation}\begin{split}\label{IU0}
\|P_t^0-\mu_0\|_{L^1(\mu_0)\to L^\infty(\mu_0)}&:=\sup_{\mu_0(|f|)\leq 1}\|P_t^0 f-\mu_0(f)\|_\infty\\
&\leq\frac{\alpha_1 \e^{-(\lambda_1-\lambda_0)t}}{(1\wedge t)^{\frac{d+2}{2}}},\quad t>0.
\end{split}\end{equation}
Then the semigroup property and the contraction property of $(P_t^0)_{t\geq0}$ in $L^p(\mu)$ for all $1\le p \le \infty$ imply that, there exists a constant $\alpha_2\geq 1$ such that
\begin{equation}\begin{split}\label{PI0}
\|P_t^0-\mu_0\|_{L^p(\mu_0)\rightarrow L^p(\mu_0)}&:=\sup_{\mu_0(|f|^p)\leq 1}\|P_t^0 f-\mu_0(f)\|_{L^p(\mu_0)}\\
&\leq \alpha_2 \e^{-(\lambda_1-\lambda_0)t},\quad t\geq 0,\,\infty \geq p\geq 1.
\end{split}\end{equation}
Applying the Riesz--Thorin interpolation theorem (see e.g. \cite[page 3]{Davies89}), by \eqref{IU0} and \eqref{PI0}, we derive that
\begin{equation}\label{PQ0}
\|P_t^0-\mu_0\|_{L^p(\mu_0)\to L^q(\mu_0)}\leq \alpha_3 \e^{-(\lambda_1-\lambda_0)t}(1\wedge t)^{-\frac{(d+2)(q-p)}{2pq}},\quad t>0,\,\infty\geq q\geq p\geq 1,
\end{equation}
for some constant $\alpha_3>0$. Thus, employing \eqref{PQ0} and \eqref{EIG0}, we can find a constant $\alpha_4>0$ such that
\begin{equation}\label{EIG0UB}
\|\phi_m\phi_0^{-1}\|_\infty\leq \alpha_4 m^{\frac{d+2}{2d}}.
\end{equation}

Now we turn to the case of subordinated Dirichlet diffusion processes. Let $B\in\mathbf{B}$. Denote the heat kernel and the semigroup corresponding to the $B$-subordinated Dirichlet diffusion process $(X_t^B)_{t\geq0}$ by $(p_t^B)_{t>0}$ and $(P_t^{D,B})_{t\geq0}$, respectively.
Using \eqref{R}, we have
\begin{equation*}\label{R1}
P_t^Df=\e^{-\lambda_0 t}\phi_0P_t^0(f\phi_0^{-1}),\quad t\geq0,\,f\in L^2(\mu_0).
\end{equation*}
Combining this with the independence of $(X_t)_{t\geq0}$ and the $B$-subordinator $(S_t^B)_{t\geq0}$, we immediately obtain that
\begin{equation}\begin{split}\label{SDSG0}
P_t^{D,B}f(x)&:=\E^x[f(X_t^B)1_{\{t<\si_\tau^B\}}]=\int_0^\infty P_s^Df(x)\,\Pp_{S_t^B}(\d s)\\
&=\int_0^\infty \e^{-\lambda_0 s}\phi_0(x)P_s^0(f\phi_0^{-1})(x)\,\Pp_{S_t^B}(\d s), \quad t>0,\,x\in M,\,f\in L^2(\mu_0),
\end{split}\end{equation}
where $\Pp_{S_t^B}$ stands for the distribution measure of $S_t^B$. By \eqref{SDSG0}, \eqref{LT}, \eqref{DHK} and \eqref{DSG}, one has that
\begin{equation}\label{SDHK}
p_t^{D,B}(x,y)=\sum_{m=0}^\infty \e^{-B(\lambda_m) t}\phi_m(x)\phi_m(y),\quad x,y\in M,\,t>0,
\end{equation}
and
\begin{equation*}\begin{split}\label{SDSG}
P_t^{D,B} f&=\int_M p_t^{D,B}(\cdot,y)f(y)\,\mu(\d y)=\sum_{m=0}^\infty \e^{-B(\lambda_m) t}\mu(\phi_m f)\phi_m,\quad t\geq0,\,f\in L^2(\mu).
\end{split}\end{equation*}
Let $\nu\in \scr{P}_0$. Then, taking $f=1$ in \eqref{SDSG0}, we deduce that
\begin{equation}\begin{split}\label{ST}
\Pp^\nu(t<\sigma_\tau^B)&=\nu\left(\int_0^\infty \e^{-\lambda_0 s}\phi_0P_s^0\phi_0^{-1}\,\Pp_{S_t^B}(\d s)\right)\\
&=\int_0^\infty \e^{-\lambda_0 s}\nu(\phi_0P_s^0\phi_0^{-1})\,\Pp_{S_t^B}(\d s),\quad t>0,
\end{split}\end{equation}
where we applied Fubini's theorem in the last equality. Since $\nu(\phi_0P_s^0\phi_0^{-1})\geq \|\phi_0\|_\infty^{-1}\nu(\phi_0)\in(0,1]$ for every $s\geq0$, by \eqref{LT}, it is easy to see that
\begin{equation}\label{SDHT-L}
\Pp^\nu(t<\sigma_\tau^B)\geq \|\phi_0\|_\infty^{-1}\nu(\phi_0)\e^{-B(\lambda_0)t},\quad t\geq0.
\end{equation}

The following facts are useful. Let $\alpha\in(0,1]$ and $B\in \BB^\alpha$. From \cite[(3.12)]{eWW}, we see that there exist constants $\kappa_1,\kappa_2>0$ and $\kappa_0\geq0$ such that
\begin{equation}\label{BR}
B(t)\geq \kappa_2 (t^\alpha\wedge t)\geq \kappa_1 t^\aa-\kappa_0,\quad t\ge 0.
\end{equation}
According to \eqref{BR}, we have
\begin{equation}\label{equ-B}
\lim_{t\rightarrow\infty}\frac{B(t-t_0)}{B(t)-B(t_0)}=1,\quad t_0\geq0.
\end{equation}
In addition,
\begin{equation}\label{SDHT}
\lim_{t\to\infty}\{\e^{B(\ll_0) t}\Pp^\nu(t<\si_\tau^B)\}=\lim_{t\to\infty}\{\e^{B(\ll_0) t}\nu(P_t^{D,B} 1)\}=\mu(\phi_0)\nu(\phi_0);
\end{equation}
see e.g. \cite[(2.20)]{LiWu2204}.

We end this section by collecting some notations. Denote by $|\cdot|$ the length in the tangent space of $M$. Let $\B_b(M)$ (resp. $\B_+(M)$)  be the class of bounded (resp. non-negative) measurable functions on $M$. Let $\Gamma(\cdot)$ be the gamma function, i.e., $\Gamma(t)=\int_0^\infty x^{t-1}\e^{-x}\,\d x$, $t>0$. We use $1_A$ to denote the index function of a set $A$. For any $x,y\in\R\cup\{\infty,-\infty\}$, $x\wedge y:=\min\{x,y\}$ and $x\vee y:=\max\{x,y\}$; in particular, $x\vee0=:x^+$.

\section{Upper bounds: proof of Theorem \ref{T1.2}}
The upper bounded estimates are based on the following general inequalities which are borrowed from \cite[Proposition 2.3]{eAST} and  \cite[Theorem 2]{Led2017} respectively (see also \cite[Theorem A.1]{eW3} for $\W_p$ for all $p\geq1$ and \cite{Peyre2018}): for every probability density function $g$ w.r.t. $\mu_0$ with $g\in L^2(\mu_0)$,
\begin{equation}\label{W2UB}\W_2(g\mu_0,\mu_0)^2\leq\int_M\frac{|\nabla (-\mathcal{L}_0)^{-1}(g-1)|^2}{\mathscr{M}(g,1)}\,\d\mu_0,\end{equation}
 and
\begin{equation}\label{W2UB-ledoux}
\W_2(f\mu_0,\mu_0)^2\leq 4\int_M|\nabla (-\mathcal{L}_0)^{-1}(f-1)|^2\,\d\mu_0,\quad f\geq0,\,\mu_0(f)=1,
\end{equation}
where for every $a,b\in[0,\infty)$,
\begin{eqnarray*}
\mathscr{M}(a,b):=
\begin{cases}
\frac{a-b}{\log a-\log b}1_{\{a\wedge b>0\}},\quad&{a\neq b},\\
\frac{1}{a}1_{\{a>0\}},\quad&{\mbox{otherwise}}.
\end{cases}
\end{eqnarray*}

Let $t>0$ and $B\in\mathbf{B}$. In order to apply the above inequalities, we consider the regularized version of $\mu_t^B$ by employing the semigroup operator $P_r^0$:
$$\mu_{t,r}^B:=\mu_t^B P_r^0,\quad r>0,$$
which means that
$$\mu_{t,r}^B(f)=\int_M P_r^0f\,\d\mu_t^B,\quad f\in \B_+(M),\,r>0.$$
Then, it is easy to see that $\mu_{t,r}^B$ is absolutely continuous with respect to $\mu_0$, and by \eqref{HK0}, we have
\beq\label{MTR}
\rho_{t,r}^B:=\frac{\d\mu_{t,r}^B}{\d\mu_0}=\frac{1}{t}\int_0^t p_r^0(X_s^B,\cdot)\,\d s=1+\sum_{m=1}^\infty \e^{-(\lambda_m-\lambda_0)r}\psi_m^B(t)\phi_m\phi_0^{-1},\quad r>0,
\end{equation}
where
$$\psi_m^B(t):=\frac{1}{t}\int_0^t\{\phi_m\phi_0^{-1}\}(X_s^B)\,\d s.$$
Note that both the families $\rho_{t,r}^B$ and $\psi_m^B(t)$ are well-defined on the event $\{t<\sigma_\tau^B\}$.

 Now we should give a brief description on the overall strategy of proof of Theorem \ref{T1.2}.
\begin{itemize}
\item[(1)] Instead of considering any initial distribution, by an approximation procedure, it suffices to deal with the initial distribution $\nu$ from $\scr{P}_0$ which is absolutely continuous w.r.t. $\mu$ such that the Radon--Nikodym derivative $\frac{\d\nu}{\d\mu}$ is controlled by $\phi_0$. See \textbf{Step 3} in the proof of Theorem \ref{T1.2}(1) for details.

\item[(2)] Applying the triangular inequality of $\W_2$, we have
\begin{equation*}\begin{split}
\E^\nu[\W_2(\mu_t^B,\mu_0)^2| t<\sigma_\tau^B]&\leq(1+\epsilon)\E^\nu[\W_2(\mu_{t,r}^B,\mu_0)^2| t<\sigma_\tau^B]\\
&\quad+(1+\epsilon^{-1})\E^\nu[\W_2(\mu_t^B,\mu_{t,r}^B)^2| t<\sigma_\tau^B],\quad \epsilon>0,\,r\in(0,1),
\end{split}\end{equation*}
where the parameter $r$ will be chosen with a suitable decay order of $t$. Then we turn to estimate respectively the two parts on the right hand side of the above inequality.

\item[(3)] The estimation of $\E^\nu[\W_2(\mu_t^B,\mu_{t,r}^B)^2| t<\sigma_\tau^B]$, which can be regarded as an error term,  is straightforward; see (3.39). Applying \eqref{W2UB-ledoux} to the dominant term $\E^\nu[\W_2(\mu_{t,r}^B,\mu_0)^2| T<\sigma_\tau^B]$, we are left to deal with $\mu(|\nabla(-\mathcal{L}_0)^{-1}(\rho_{t,r}^B-1)|^2)$, which needs much effort; see Lemma \ref{L3.2}(2).  Indeed, this strategy is effective to derive rates of convergence in Theorem \ref{T1.2}. However, in order to obtain the constant $c=1$ in the second assertion of Theorem \ref{T1.2}(1), in the case when $d<6\alpha-2$ with some $\alpha\in(\frac{1}{2},1]$, we should apply \eqref{W2UB} in a different manner, and more effort is involved to estimate the resultant term due to the extra factor corresponding to the right hand side of \eqref{W2UB}; see Lemmas \ref{L3.2}(1), \ref{L3.6} and \ref{L3.7} and \eqref{L3.8+1}. Moreover, the establishment of Lemmas \ref{L3.6} and \ref{L3.7} in turn depends on Lemma \ref{L3.5} and Lemma \ref{L3.4}, respectively.
\end{itemize}

In the sequel, we will also frequently use the notation $\nu_0$, i.e.,
$$\nu_0:=\ff{\phi_0}{\mu(\phi_0)}\mu,$$
and apply the following result.
\begin{lem}\label{nu0}
Let $B\in\mathbf{B}$ and $\nu\in\scr{P}_0$ such that $\nu=h\mu$ with $|h\phi_0^{-1}|\le C$ for some constant $C>0$. Then there exists a constant $K>0$ such that
\begin{equation*}\label{12L3.2}
\E^\nu(g|t<\si_\tau^B)\leq K\E^{\nu_0}(g|t<\si_\tau^B),\quad g\in\B_+(M),\,t\ge 0.
\end{equation*}
\end{lem}
\begin{proof} Let $g\in\B_+(M)$. It is clear that
\begin{equation*}\begin{split}\label{11L3.2}
\E^\nu[g]&=\int_{M}\E^x[g]\,\nu(\d x)=\int_M h(x)\E^x[g]\,\mu(\d x)\\
&\le C\int_M \phi_0(x)\E^x[g]\,\mu(\d x)=C\mu(\phi_0)\E^{\nu_0}[g].
\end{split}\end{equation*}
Then
\begin{equation}\label{12L3.2}
\E^\nu(g|t<\si_\tau^B)=\ff{\E^\nu(1_{\{t<\si_\tau^B\}}g)}{\Pp^\nu(t<\si_\tau^B)}
\leq\ff{C\mu(\phi_0)\E^{\nu_0}(1_{\{t<\si_\tau^B\}}g)}{\Pp^{\nu}(t<\si_\tau^B)},\quad t\ge 0.
\end{equation}

Since $\mu(\phi_0)\in(0,\infty)$, it suffices to show that there exists a constant $c>0$ such that
$$\Pp^\nu(t<\si_\tau^B)\ge c\Pp^{\nu_0}(t<\si_\tau^B),\quad t\ge 0.$$
By the proof of \cite[Lemma 3.2(2)]{eW2}, we have
$$\Pp^{\nu}(t<\tau)\ge c_0\Pp^{\nu_0}(t<\tau),\quad t\ge 0,$$
for some constant $c_0>0$. Then, by \eqref{SDSG0}, we see that
\begin{equation*}\begin{split}
\Pp^{\nu}(t<\si_\tau^B)&=\int_0^\infty \P^\nu(s<\tau)\,\P_{S_t^B}(\d s)\\
&\ge c_0\int_0^\infty \P^{\nu_0}(s<\tau)\,\P_{S_t^B}(\d s)\\
&= c_0\Pp^{\nu_0}(t<\si_\tau^B),\quad t\geq0.
\end{split}\end{equation*}
The proof is completed.
\end{proof}

\begin{lem}\label{L3.2}
Let $\aa\in(0,1]$, $B\in \BB^\aa$ and $\nu\in\scr{P}_0$.  Suppose that $\nu=h\mu$ with $\|h\phi_0^{-1}\|_\infty<\infty$.
\begin{itemize}
\item[$(1)$] If $\alpha\in(\frac{1}{2},1]$ and $d<6\alpha-2$,
then there exists a constant $c>0$ such that
\begin{equation*}\begin{split}
&\sup_{T\geq t}\left|t\E^\nu\big[\mu_0\big(|\nabla (-\mathcal{L}_0)^{-1}(\rho_{t,r}^B-1)|^2\big)|T<\sigma_\tau^B\big]-
2\sum_{m=1}^\infty\frac{\e^{-2(\lambda_m-\lambda_0)r}}{(\lambda_m-\lambda_0)(B(\lambda_m)-B(\lambda_0))}\right|\\
&\leq c t^{-1}\Big(r^{-\frac{(d-2)^+}{2}}+1_{\{d=2\}}\log r^{-1}\Big),\quad r\in(0,1],\,t\geq 1.
\end{split}\end{equation*}
\item[$(2)$] If $d\ge 6\aa-2$,  then for any $\gamma>\frac{d+2-6\alpha}{6}$, there exists a constant $c>0$ such that for any $t\ge 1$ and $r>0$,
\begin{equation}\begin{split}\label{0L3.2}
&\sup_{T\geq t}t\E^\nu\big[\mu_0\big(|\nabla (-\mathcal{L}_0)^{-1}(\rho_{t,r}^B-1)|^2\big)|T<\sigma_\tau^B\big]\\
&\le \sum_{m=1}^\infty\ff{c\e^{-2(\ll_m-\ll_0)r}}{(\ll_m-\ll_0)[B(\ll_m)-B(\ll_0)]}+\ff{cr^{-\gg}}{t}\sum_{m=1}^\infty
\ff{\e^{-(\ll_m-\ll_0)r}}{(\ll_m-\ll_0)[B(\ll_m)-B(\ll_0)]}.
\end{split}\end{equation}
Moreover, if $d\geq 2(1+\alpha)$,  then for any $\gamma>\frac{d+2-6\alpha}{6}$, there exists a constant $C>0$ such that
\begin{equation}\begin{split}\label{0L3.3}
&\sup_{T\geq t}t\E^\nu\big[\mu_0\big(|\nabla (-\mathcal{L}_0)^{-1}(\rho_{t,r}^B-1)|^2\big)|T<\sigma_\tau^B\big]\\
&\leq C\Big(r^{-\frac{d-2(1+\alpha)}{2}}+1_{\{d=2(1+\alpha)\}}\log r^{-1}+t^{-1}r^{-\gamma}\Big),\quad r\in(0,1),\, t\geq 1.
\end{split}\end{equation}
\end{itemize}
\end{lem}
\begin{proof}
Let $t,r>0$ and $T\geq t$. Since $(\phi_n)_{n\in\mathbb{N}_0}$ is an orthonormal basis in $L^2(\mu)$, by \eqref{EIG0} and the integration-by-parts formula, we see that $\{\ff {\nn(\phi_m\phi_0^{-1})}{\ss{\ll_m-\ll_0}}\}_{m\in\mathbb{N}}$ is an orthonormal basis in $\overrightarrow{L}^2(\mu_0)$, the class of measurable vector fields $v$ on $M$ such that $|v|\in L^2(\mu_0)$. Then
\begin{equation*}\begin{split}
&\mu_0(|\nn (-\mathcal{L}_0)^{-1}(\rho_{t,r}^B-1)|^2)\\
&=\mu_0\left(\left|\nn (-\mathcal{L}_0)^{-1}\sum_{m=1}^\infty \e^{-(\ll_m-\ll_0)r}\psi_m^B(t)(\phi_m\phi_0^{-1})\right|^2\right)\\
&=\mu_0\left(\left|\sum_{m=1}^\infty\ff {\psi_m^B(t)}{(\ll_m-\ll_0)\e^{(\ll_m-\ll_0)r}}\nn(\phi_m\phi_0^{-1})\right|^2\right)\\
&=\sum_{m=1}^\infty\ff{|\psi_m^B(t)|^2}{(\ll_m-\ll_0)\e^{2(\ll_m-\ll_0)r}}.
\end{split}\end{equation*}
Hence
\begin{equation}\begin{split}\label{1L3.2}
&t\E^\nu\big[\mu_0\big(|\nabla(-\mathcal{L}_0)^{-1}(\rho_{t,r}^B-1)|^2\big)|T<\sigma_\tau^B\big]\\
&=\sum_{m=1}^\infty\frac{t\E^\nu[|\psi_m^B(t)|^2|T<\sigma_\tau^B]}{\lambda_m-\lambda_0}\e^{-2(\lambda_m-\lambda_0)r}\\
&=\sum_{m=1}^\infty\frac{2\int_0^t \d s_1\int_{s_1}^t\E^\nu[1_{\{T<\sigma_\tau^B\}}(\phi_m\phi_0^{-1})(X_{s_1}^B)(\phi_m\phi_0^{-1})(X_{s_2}^B)]\,\d s_2}{t\e^{2(\lambda_m-\lambda_0)r}(\lambda_m-\lambda_0)\Pp^\nu(T<\sigma_\tau^B)}.
\end{split}\end{equation}

Let $f\in\B_b(M)$. Then, for any $0\leq s\leq T$, by the Markov property,
\begin{equation*}\begin{split}
\E^x[f(X_s^B)1_{\{T<\sigma_\tau^B\}}]&=\E^x\big\{1_{\{s<\sigma_\tau^B\}}f(X_s^B)\E^{X_s^B}[1_{\{T-s<\sigma_\tau^B\}}]\big\}\\
&=P_s^{D,B}\{fP_{T-s}^{D,B} 1\}(x),
\end{split}\end{equation*}
which together with \eqref{SDSG0}
implies that
\begin{equation}\begin{split}\label{2L3.2}
\E^{\nu}[f(X_s^B)1_{\{T<\sigma_\tau^B\}}]
=\int_0^\infty\!\!\!\int_0^\infty \e^{-\lambda_0 (u+v)}\nu(\phi_0P_v^0\{fP_u^0\phi_0^{-1}\})\,\Pp_{S_{T-s}^B}(\d u)\Pp_{S_{s}^B}(\d v).
\end{split}\end{equation}
Furthermore, we have that, for any $0<s_1<s_2\leq T$,
\begin{equation}\begin{split}\label{3L3.2}
&\E^\nu[f(X_{s_1}^B)f(X_{s_2}^B)1_{\{T<\sigma_\tau^B\}}]\\
&=\int_M\E^x\big[1_{\{s_1<\sigma_\tau^B\}}f(X_{s_1}^B)\E^{X_{s_1}^B}\{f(X_{s_2-s_1}^B)1_{\{T-s_1<\sigma_\tau^B\}}\}\big]\,\nu(\d x)\\
&=\int_M\E^x\big(1_{\{s_1<\sigma_\tau^B\}}f(X_{s_1}^B)\E^{X_{s_1}^B}\big\{f(X_{s_2-s_1}^B)1_{\{s_2-s_1<\si_\tau^B\}}
\E^{X_{s_2-s_1}^B}[1_{\{T-s_2<\sigma_{\tau}^B\}}]\big\}\big)\,\nu(\d x)\\
&=\nu\big(P_{s_1}^{D,B}\big[fP_{s_2-s_1}^{D,B}(f P_{T-s_2}^{D,B}1)\big]\big)\\
&=\int_0^\infty\!\!\! \int_0^\infty \!\!\!\int_0^\infty \e^{-\lambda_0 (u+v+w)}
\nu\big(\phi_0 P_u^0\big[fP_v^0(fP_w^0\phi_0^{-1})\big]\big)\,\Pp_{S_{T-s_2}^B}(\d w)\Pp_{S_{s_2-s_1}^B}(\d v)\Pp_{S_{s_1}^B}(\d u).
\end{split}\end{equation}

Since $\nu=h\phi_0^{-2}\mu_0$, by the symmetry of $(P^0_t)_{t\geq0}$ in $L^2(\mu_0)$, we derive that, for every $u,v,w>0$,
\begin{equation*}\begin{split}
&\nu\big(\phi_0P_u^0\big[fP_v^0(fP_w^0\phi_0^{-1})\big]\big)=\mu_0\big(h\phi_0^{-1}P_u^0\big[fP_v^0(fP_w^0\phi_0^{-1})\big]\big)\\
&=\mu_0\big(P_u^0(h\phi_0^{-1})fP_v^0(fP_w^0\phi_0^{-1})\big)\\
&=\mu_0\big(P_u^0(h\phi_0^{-1})fP_v^0\{f[\mu(\phi_0)+P_w^0(\phi_0^{-1})-\mu(\phi_0)]\}\big)\\
&=\mu(\phi_0)\mu_0\big(P_u^0(h\phi_0^{-1})fP_v^0f\big)+\mu_0\big(P_u^0(h\phi_0^{-1})fP_v^0\{f[P_w^0(\phi_0^{-1})-\mu(\phi_0)]\}\big)\\
&=:{\rm I_1}+{\rm I_2}.
\end{split}\end{equation*}
In particular,  for each $m\in\mathbb{N}$, taking $f=\phi_m\phi_0^{-1}$ and noting that $\|\phi_m\phi_0^{-1}\|_{L^2(\mu_0)}=1$ and $P_t^0 (\phi_m\phi_0^{-1})=\e^{-(\ll_m-\ll_0)t}\phi_m\phi_0^{-1}$ for every $t>0$, we have, for any $u,v,w>0$ and any $m\in\mathbb{N}$,
\begin{equation*}\begin{split}
{\rm I_1}&=\mu(\phi_0)\mu_0[P_u^0(h\phi_0^{-1})\phi_m\phi_0^{-1}P_v^0(\phi_m\phi_0^{-1})]\\
&=\mu(\phi_0)\mu_0\big([P_u^0(h\phi_0^{-1})-\mu(h\phi_0)+\mu(h\phi_0)]\e^{-(\lambda_m-\lambda_0)v}\{\phi_m\phi_0^{-1}\}^2\big)\\
&=\mu(\phi_0)\e^{-(\lambda_m-\lambda_0)v}\mu_0\big(\{\phi_m\phi_0^{-1}\}^2[P_u^0(h\phi_0^{-1})-\mu(h\phi_0)]\big)\\
&\quad+\mu(\phi_0)\mu(h\phi_0)\e^{-(\lambda_m-\lambda_0)v},
\end{split}\end{equation*}
and
\begin{equation*}\begin{split}
{\rm I_2}&=\mu_0\big(P_u^0(h\phi_0^{-1})\phi_m\phi_0^{-1}P_v^0\{\phi_m\phi_0^{-1}[P_w^0\phi_0^{-1}-\mu(\phi_0)]\}\big)\\
&=\mu_0\big([P_u^0(h\phi_0^{-1})-\mu(h\phi_0)+\mu(h\phi_0)]\phi_m\phi_0^{-1}P_v^0\{\phi_m\phi_0^{-1}[P_w^0\phi_0^{-1}-\mu(\phi_0)]\}\big)\\
&=\mu_0\big([P_u^0(h\phi_0^{-1})-\mu(h\phi_0)]\phi_m\phi_0^{-1}P_v^0\{\phi_m\phi_0^{-1}[P_w^0\phi_0^{-1}-\mu(\phi_0)]\}\big)\\
&\quad+\mu(h\phi_0)\mu_0\big(\phi_m\phi_0^{-1}P_v^0\{\phi_m\phi_0^{-1}[P_w^0\phi_0^{-1}-\mu(\phi_0)]\}\big)\\
&=\mu_0\big([P_u^0(h\phi_0^{-1})-\mu(h\phi_0)]\phi_m\phi_0^{-1}P_v^0\{\phi_m\phi_0^{-1}[P_w^0\phi_0^{-1}-\mu(\phi_0)]\}\big)\\
&\quad+\mu(h\phi_0)\e^{-(\lambda_m-\lambda_0)v}\mu_0\big(\{\phi_m\phi_0^{-1}\}^2\{P_w^0\phi_0^{-1}-\mu(\phi_0)\}\big).
\end{split}\end{equation*}
Thus, for every $u,v,w>0$ and each $m\in\mathbb{N}$, we can divide $\nu\big(\phi_0 P_u^0[\phi_m\phi_0^{-1}P_v^0(\phi_m\phi_0^{-1}P_w^0\phi_0^{-1})]\big)$ into the following four parts, i.e.,
\begin{equation}\begin{split}\label{5L3.2}
&\nu\big(\phi_0 P_u^0[\phi_m\phi_0^{-1}P_v^0(\phi_m\phi_0^{-1}P_w^0\phi_0^{-1})]\big)\\
&=\nu(\phi_0)\mu(\phi_0)\e^{-(\lambda_m-\lambda_0)v}+{\rm J_1}(u,v,w)+{\rm J_2}(u,v,w)+{\rm J_3}(u,v,w),
\end{split}\end{equation}
where
\begin{align*}
&{\rm J_1}(u,v,w):=\mu_0\big([P_u^0(h\phi_0^{-1})-\mu(h\phi_0)]\phi_m\phi_0^{-1}P_v^0\{\phi_m\phi_0^{-1}[P_w^0\phi_0^{-1}-\mu(\phi_0)]\}\big),\\
&{\rm J_2}(u,v,w):=\mu(\phi_0)\e^{-(\lambda_m-\lambda_0)v}\mu_0\big(\{\phi_m\phi_0^{-1}\}^2 [P_u^0(h\phi_0^{-1})-\mu(h\phi_0)]\big),\\
&{\rm J_3}(u,v,w):=\mu(h\phi_0)\e^{-(\lambda_m-\lambda_0)v}\mu_0\big(\{\phi_m\phi_0^{-1}\}^2\{P_w^0\phi_0^{-1}-\mu(\phi_0)\}\big).
\end{align*}

By \eqref{PI0}, the symmetry of $(P_t^0)_{t\geq0}$ in $L^2(\mu_0)$ and $\|\phi_0^{-1}\|_{L^2(\mu_0)}=1$, there exists a constant $c>0$ such that
\begin{equation*}\begin{split}
&|\nu(\phi_0P_t^0\phi_0^{-1})-\mu(\phi_0)\nu(\phi_0)|=|\mu_0(h\phi_0^{-1}P_t^0\phi_0^{-1})-\mu(\phi_0)\mu_0(h\phi_0^{-1})|\\
&=|\mu_0\big(h\phi_0^{-1}[P_t^0-\mu_0]\phi_0^{-1}\big)|\leq \|h\phi_0^{-1}\|_\infty\mu_0(|[P_t^0-\mu_0]\phi_0^{-1}|)\\
&\leq \|h\phi_0^{-1}\|_\infty\|(P_t^0-\mu_0)\phi_0^{-1}\|_{L^2(\mu_0)}\leq \|h\phi_0^{-1}\|_\infty\|P_t^0-\mu_0\|_{L^2(\mu_0)\rightarrow L^2(\mu_0)}\|\phi_0^{-1}\|_{L^2(\mu_0)}\\
&\leq c \e^{-(\lambda_1-\lambda_0)t},\quad t>0,
\end{split}\end{equation*}
which together with \eqref{LT} leads to
\begin{equation}\begin{split}\label{7L3.2}
&\left|\int_0^\infty\e^{-\lambda_0s}\big[\nu(\phi_0P_s^0\phi_0^{-1})-\mu(\phi_0)\nu(\phi_0)\big]\,\Pp_{S_T^B}(\d s)\right|\\
&\leq c \int_0^\infty\e^{-\lambda_0 s}\e^{-(\lambda_1-\lambda_0)s}\,\Pp_{S_T^B}(\d s)\\
&=c\e^{-B(\lambda_1) T},\quad T>0.
\end{split}\end{equation}

By \eqref{1L3.2}, \eqref{3L3.2} with $f$ replaced by $\phi_m\phi_0^{-1}$, \eqref{5L3.2}, \eqref{7L3.2}, \eqref{SDHT-L}, and the elementary fact that
\begin{equation*}\label{6L3.2}
\int_0^t \d s_1 \int_{s_1}^t  \e^{-a(s_2-s_1)}\,\d s_2=\frac{t}{a}-\frac{1-\e^{-at}}{a^2},\quad a>0,\,t\geq0,
\end{equation*}
we can find a constant $C>0$ such that, for every $t>0$
and every $r\in(0,1]$,
\begin{equation}\begin{split}\label{8L3.2}
&\sup_{T\geq t}\left|t\E^\nu\big[\mu_0\big(|\nabla (-\mathcal{L}_0)^{-1}(\rho_{t,r}^B-1)|^2\big)|T<\sigma_\tau^B\big]-2\sum_{m=1}^\infty\frac{\e^{-2(\lambda_m-\lambda_0)r}}{(\lambda_m-\lambda_0)
[B(\lambda_m)-B(\lambda_0)]}\right|\\
&\leq\frac{C}{t}\sum_{m=1}^\infty\left(\frac{\e^{-2(\lambda_m-\lambda_0)r}}{(\lambda_m-\lambda_0)[B(\lambda_m)-B(\lambda_0)]}
+\frac{\e^{-2(\lambda_m-\lambda_0)r}}{\P^\nu(T<\si_\tau^B)(\lambda_m-\lambda_0)}\left|\int_0^t\d s_1 \int_{s_1}^t\Xi_T(s_1,s_2)\,\d s_2\right|\right),
\end{split}\end{equation}
where, for every $0<s_1<s_2\leq T$, we set
\begin{equation}\label{Xi}
\Xi_T(s_1,s_2):=\int_0^\infty\!\!\!\int_0^\infty\!\!\!\int_0^\infty \!\!\e^{-\lambda_0 (u+v+w)}({\rm J_1}+{\rm J_2}+{\rm J_3})(u,v,w)\,\Pp_{S_{T-s_2}^B}(\d w)\Pp_{S_{s_2-s_1}^B}(\d v)\Pp_{S_{s_1}^B}(\d u).
\end{equation}
(For the convenience of the reader, we include a detailed proof of \eqref{8L3.2} in Appendix.)

\textbf{(1)} Let $\alpha\in(\frac{1}{2},1]$ and $d<6\alpha-2$. Due to \eqref{PHI}, the fact that$\|\phi_m\phi_0^{-1}\|_{L^2(\mu_0)}=1$ for each $m\in\mathbb{N}$, \eqref{PI0} and \eqref{PQ0}, we deduce that, for any $\theta\in[1,3)$,
\begin{equation}\begin{split}\label{J1}
|{\rm J_1}(u,v,w)|&\le\|P_u^0(h\phi_0^{-1})-\mu_0(h\phi_0^{-1})\|_\infty\|P_w^0[\phi_0^{-1}-\mu_0(\phi_0^{-1})]\|_\infty
\mu_0\big(|\phi_m\phi_0^{-1}|P_v^0|\phi_m\phi_0^{-1}|\big)\\
&=\|P_u^0(h\phi_0^{-1})-\mu_0(h\phi_0^{-1})\|_\infty\|P_w^0[\phi_0^{-1}-\mu_0(\phi_0^{-1})]\|_\infty
\mu_0\big([P_{\ff v 2}^0|\phi_m\phi_0^{-1}|]^2\big)\\
&\le\|P_u^0(h\phi_0^{-1})-\mu_0(h\phi_0^{-1})\|_\infty\|P_w^0[\phi_0^{-1}-\mu_0(\phi_0^{-1})]\|_\infty
\mu_0\big(P_{\ff v 2}^0|\phi_m\phi_0^{-1}|^2\big)\\
&=\|P_u^0(h\phi_0^{-1})-\mu_0(h\phi_0^{-1})\|_\infty\|P_w^0[\phi_0^{-1}-\mu_0(\phi_0^{-1})]\|_\infty\\
&\leq \|h\phi_0^{-1}\|_\infty\|P_u^0-\mu_0\|_{L^\infty(\mu_0)\to L^\infty(\mu_0)} \|\phi_0^{-1}\|_{L^\theta(\mu_0)}\|P_w^0-\mu_0\|_{L^\theta(\mu_0)\to L^\infty(\mu_0)}\\
&\leq c_1 \e^{-(\lambda_1-\lambda_0)(u+w)}\{1\wedge w\}^{-\frac{d+2}{2\theta}},\quad u,v,w>0,
\end{split}\end{equation}
for some constant $c_1>0$, where in the first and the second equalities we also used the symmetry and the invariance of $(P_t^0)_{t\geq0}$ w.r.t. $\mu_0$,  respectively. Analogously, we may estimate ${\rm J_1}+{\rm J_2}$ as follows: for any $\theta\in[1,3)$,
\begin{equation}\begin{split}\label{J2J3}
&|{\rm J_2}+{\rm J_3}|(u,v,w)\\
&\leq c_2 \e^{-(\lambda_m-\lambda_0)v}\big(\|P_u^0-\mu_0\|_{L^\infty(\mu_0)\to L^\infty(\mu_0)}+\|P_w^0-\mu_0\|_{L^\theta(\mu_0)\to L^\infty(\mu_0)}\big)\\
&\leq c_3 \e^{-(\lambda_m-\lambda_0)v}\left(\e^{-(\lambda_1-\lambda_0)u}+\{1\wedge w\}^{-\frac{d+2}{2\theta}}\e^{-(\lambda_1-\lambda_0)w}\right),\quad u,v,w>0,\,m\in\mathbb{N},
\end{split}\end{equation}
for some constants $c_2,c_3>0$.

By \eqref{BR}, we find a constant $K>0$ such that
\begin{align*}
\int_0^\infty(1\wedge t)^{-\kappa}\,\Pp_{S_r^B}(\d t)
&\le 1 +   \E\big[   (S_r^B)^{-\kappa} \big] \\
&= 1 +\ff 1 {\GG(\kappa)} \int_0^\infty t^{\kappa-1} \e^{-r B(t)}\d t \\
& \leq K\big(1+r^{-\ff \kappa \aa}\big),\quad \kappa>0,\,r>0.
\end{align*}

Then, combining this with \eqref{J1} and \eqref{J2J3}, we deduce that there exist some constants $c_4,c_5,c_6,c_7>0$ such that, for every $0<s_1<s_2\leq T$,
\begin{equation}\begin{split}\label{9L3.2}
&\int_0^\infty\!\!\!\int_0^\infty\!\!\!\int_0^\infty \e^{-\lambda_0 (w+v+ u)}|{\rm J}_1(u,v,w)|\,\Pp_{S_{T-s_2}^B}(\d w)\Pp_{S_{s_2-s_1}^B}(\d v)\Pp_{S_{s_1}^B}(\d u)\\
&\leq c_4 \int_0^\infty\!\!\!\int_0^\infty\!\!\!\int_0^\infty \e^{-\lambda_0 v}\e^{-\lambda_1(u+w)}\{1\wedge w\}^{-\frac{d+2}{2\theta}}\,\Pp_{S_{T-s_2}^B}(\d w)\Pp_{S_{s_2-s_1}^B}(\d v)\Pp_{S_{s_1}^B}(\d u)\\
&=c_4 \int_0^\infty \e^{-\lambda_0 v}\,\Pp_{S_{s_2-s_1}^B}(\d v) \int_0^\infty \e^{-\lambda_1 u}\,\Pp_{S_{s_1}^B}(\d u)\int_0^\infty \e^{-\lambda_1 w}\{1\wedge w\}^{-\frac{d+2}{2\theta}}\,\Pp_{S_{T-s_2}^B}(\d w)\\
&\leq c_5 \e^{-B(\lambda_0)(s_2-s_1)}\e^{-B(\lambda_1) s_1}\e^{-B(\lambda_1) (T-s_2)}\Big[1+(T-s_2)^{-\frac{d+2}{2\alpha\theta}}\Big],
\end{split}\end{equation}
and
\begin{equation}\begin{split}\label{10L3.2}
&\int_0^\infty\!\!\!\int_0^\infty\!\!\!\int_0^\infty \e^{-\lambda_0 (w+v+ u)}|{\rm J}_2+{\rm J}_3|(u,v,w)\,\Pp_{S_{T-s_2}^B}(\d w)\Pp_{S_{s_2-s_1}^B}(\d v)\Pp_{S_{s_1}^B}(\d u)\\
&\leq c_6 \int_0^\infty\!\!\!\int_0^\infty\!\!\!\int_0^\infty \e^{-\lambda_0 w}\e^{-\lambda_m v}\e^{-\lambda_1 u} \,\Pp_{S_{T-s_2}^B}(\d w)\Pp_{S_{s_2-s_1}^B}(\d v)\Pp_{S_{s_1}^B}(\d u)\\
&\quad+c_6 \int_0^\infty\!\!\!\int_0^\infty\!\!\!\int_0^\infty \e^{-\lambda_0 u}\e^{-\lambda_m v}\e^{-\lambda_1 w} \{1\wedge w\}^{-\frac{d+2}{2\theta}}\,\Pp_{S_{T-s_2}^B}(\d w)\Pp_{S_{s_2-s_1}^B}(\d v)\Pp_{S_{s_1}^B}(\d u)\\
&=c_6 \e^{-B(\lambda_0) (T-s_2)}\e^{-B(\lambda_m)(s_2-s_1)}\e^{-B(\lambda_1)s_1}\\
&\quad+c_6 \e^{-B(\lambda_0) s_1}\e^{-B(\lambda_m)(s_2-s_1)}\int_0^\infty \e^{-\lambda_1 w}\{1\wedge w\}^{-\frac{d+2}{2\theta}}\,\Pp_{S_{T-s_2}^B}(\d w)\\
&\leq c_7\e^{-B(\lambda_0) (T-s_2)}\e^{-B(\lambda_m)(s_2-s_1)}\e^{-B(\lambda_1) s_1} \\
&\quad+c_7\e^{-B(\lambda_0) s_1}\e^{-B(\lambda_m)(s_2-s_1)}\e^{-B(\lambda_1)(T-s_2)}\Big(1+(T-s_2)^{-\frac{d+2}{2\alpha\theta}}\Big),\quad m\in\mathbb{N}.
\end{split}\end{equation}

Taking $\theta\in(\frac{d+2}{2\alpha},3)$, we have $0<\frac{d+2}{2\alpha\theta}<1$ since $d<6\alpha-2$.
By \eqref{9L3.2} and \eqref{10L3.2} and a careful computation, there exists a constant $C_1>0$ such that
\begin{equation}\begin{split}\label{10L3.2+}
\e^{B(\ll_0)T}\int_0^t \d s_1\int_{s_1}^t\Xi_T(s_1,s_2)\,\d s_2\leq C_1
\end{split}\end{equation}
holds for every $T\geq t>0$ and each $m\in\mathbb{N}$; see also the detailed proof in Appendix. Then, we infer from \eqref{8L3.2}, \eqref{BR}, \eqref{equ-B} and \eqref{EIG} that there are constants $C_2,C_3,C_4,C_5>0$ so that, for any $t\ge 1$  and  any $r\in(0,1]$,
\begin{equation*}\begin{split}\label{IL3.2}
&\sup_{T\geq t}\left|t\E^\nu\big[\mu_0\big(|\nabla (-\mathcal{L}_0)^{-1}(\rho_{t,r}^B-1)|^2\big)|T<\sigma_\tau^B\big]-2\sum_{m=1}^\infty
\frac{\e^{-2(\lambda_m-\lambda_0)r}}{(\lambda_m-\lambda_0)(B(\lambda_m)-B(\lambda_0))}\right|\\
&\leq \frac{C_2}{t}\sum_{m=1}^\infty \frac{\e^{-2(\lambda_m-\lambda_0)r}}{\lambda_m-\lambda_0}\leq\frac{C_3}{t}\int_1^\infty s^{-\frac{2}{d}}\e^{-C_4 r s^{2/d}}\,\d s\\
&\leq C_{5} t^{-1}\Big(r^{-\frac{(d-2)^+}{2}}+1_{\{d=2\}}\log r^{-1}\Big).
\end{split}\end{equation*}

\textbf{(2)} Let $\alpha\in(0,1]$ and $d\ge 6\aa-2$. Set $g:=\mu_0(|\nabla (-\mathcal{L}_0)^{-1}(\rho_{t,r}^B-1)|^2)$. Then by Lemma \ref{nu0},
there exists some constant $\kappa_0>0$ such that
\begin{equation*}\label{12L3.2}
\E^\nu(g|t<\si_\tau^B)\leq\kappa_0\E^{\nu_0}(g|t<\si_\tau^B),\quad t\ge 0.
\end{equation*}
So, we only need to prove that the desired assertion holds for $\nu=\nu_0$. In that case, we observe that $h\phi_0^{-1}=\mu(\phi_0)^{-1}$ is a constant, which implies that $J_1=J_2=0$. Next, we should deal with $J_3$.

By \eqref{EIG0}, \eqref{PQ0}, H\"{o}lder's inequality, $\mu_0(\phi_m\phi_0^{-1})=0$ and $\|\phi_m\phi_0^{-1}\|_{L^2(\mu_0)}=1$ for each $m\in\mathbb{N}$,  and $\|\phi_0\|_{L^\theta(\mu_0)}<\infty$ for every $\theta\in[1,3\wedge \frac{p}{p-1})$ and every $p\geq1$, we find a constant $c_1>0$ such that
\begin{equation*}\begin{split}
&|{\rm J_3}(u,v,w)|\\
&=\big|\mu_0(h\phi_0^{-1})\e^{-(\ll_m-\ll_0)v}\mu_0\big(\{\e^{(\ll_m-\ll_0)\ff r 2}[P_{\ff r 2}^0-\mu_0](\phi_m\phi_0^{-1})\}^2\{P_w^0\phi_0^{-1}-\mu(\phi_0)\}\big)\big|\\
&\leq \|h\phi_0^{-1}\|_\infty \e^{-(\lambda_m-\lambda_0)(v-r)}\big\||(P_{\frac{r}{2}}^0-\mu_0)(\phi_m\phi_0^{-1})|^2(P_w^0-\mu_0)\phi_0^{-1}\big\|_{L^1(\mu_0)}\\
&\le \|h\phi_0^{-1}\|_\infty \e^{-(\lambda_m-\lambda_0)(v-r)}\big\||(P_{\frac{r}{2}}^0-\mu_0)(\phi_m\phi_0^{-1})|^2\|_{L^p(\mu_0)}\|(P_w^0-\mu_0)\phi_0^{-1}\big\|_{L^{\ff p {p-1}}(\mu_0)}\\
&\leq \|h\phi_0^{-1}\|_\infty  \e^{-(\lambda_m-\lambda_0)(v-r)}\|P_{\frac{r}{2}}^0-\mu_0\|_{L^{2}(\mu_0)\to L^{2p}(\mu_0)}^2\|P_w^0-\mu_0\|_{L^\theta(\mu_0)\to L^{\frac{p}{p-1}}(\mu_0)}\|\phi_0^{-1}\|_{L^\theta(\mu_0)}\\
&\leq c_1 \e^{-(\lambda_m-\lambda_0)(v-r)-(\lambda_1-\lambda_0)w}r^{-\frac{(d+2)(p-1)}{2p}}\{1\wedge w\}^{-\frac{(d+2)[\theta-(\theta-1)p]}{2\theta p}},\quad u,v,w>0,\,r\in(0,1).
\end{split}\end{equation*}
Hence, by a similar deduction for \eqref{9L3.2}, there exists some constant $c_2>0$ such that, for every $p\in[1,\frac{\theta}{\theta-1}]$ and each $m\in\mathbb{N}$,
\begin{equation}\begin{split}\label{J3}
&\int_0^\infty\!\!\!\int_0^\infty\!\!\!\int_0^\infty \e^{-\lambda_0 (u+v+w)}|{\rm J_3}(u,v,w)|\,\Pp_{S_{T-s_2}^B}(\d w)\Pp_{S_{s_2-s_1}^B}(\d v)\Pp_{S_{s_1}^B}(\d u)\\
&\leq c_1 r^{-\frac{(d+2)(p-1)}{2p}}\e^{(\lambda_m-\lambda_0)r}\\
&\quad \times\int_0^\infty\!\!\!\int_0^\infty\!\!\!\int_0^\infty \e^{-\lambda_0 u}\e^{-\lambda_m v}\e^{-\lambda_1 w}[1\wedge w]^{-\frac{(d+2)[\theta-(\theta-1)p]}{2\theta p}}\Pp_{S_{T-s_2}^B}(\d w)\Pp_{S_{s_2-s_1}^B}(\d v)\Pp_{S_{s_1}^B}(\d u)\\
&=c_1 r^{-\frac{(d+2)(p-1)}{2p}}\e^{(\lambda_m-\lambda_0)r} \e^{-B(\lambda_0) s_1}\e^{-B(\lambda_m)(s_2-s_1)}\int_0^\infty \e^{-\lambda_1 w}[1\wedge w]^{-\frac{(d+2)[\theta-(\theta-1)p]}{2\theta p}}\,\Pp_{S_{T-s_2}^B}(\d w)\\
&\leq c_2 r^{-\frac{(d+2)(p-1)}{2p}}\e^{(\lambda_m-\lambda_0)r}\e^{[B(\lambda_m)-B(\lambda_0)]s_1}\e^{-B(\lambda_m) s_2}\e^{-B(\lambda_1)(T-s_2)}\Big((T-s_2)^{-\frac{(d+2)[\theta-(\theta-1)p]}{2\alpha\theta p}}+1\Big).
\end{split}\end{equation}
Let $p_0:=\frac{3(d+2)}{2d+4+6\alpha}$. Obviously, $p_0\ge 1$ when $d\ge 6\aa-2$. Note that $p\mapsto \frac{(d+2)(p-1)}{2p}$
is increasing, $\theta\mapsto\frac{(d+2)[\theta-(\theta-1)p]}{2\alpha\theta p}$ is deceasing, and
\begin{align*}
&\lim_{p\downarrow p_0}\frac{(d+2)(p-1)}{2p}=\frac{d+2-6\alpha}{6},\\
&\lim_{\theta\uparrow 3}\frac{(d+2)[\theta-(\theta-1)p]}{2\alpha\theta p}=\frac{(d+2)(3-2p)}{6\alpha p}\in(0,1),\quad p\in(p_0,3/2).
\end{align*}
Hence, for any $\gamma>\frac{d+2-6\alpha}{6}$, there exist $\theta_\ast\in[1,3)$ and $p_\ast\in(p_0,\frac{\theta_*}{\theta_*-1})$
such that
$$\frac{(d+2)(p_\ast-1)}{2p_\ast}\leq \gamma,\quad \delta:=\frac{(d+2)[\theta_\ast-(\theta_\ast-1)p_\ast]}{2\alpha\theta_\ast p_\ast}<1.$$
Thus, by \eqref{J3},
\begin{equation*}\begin{split}
&\int_0^\infty\!\!\!\int_0^\infty\!\!\!\int_0^\infty \e^{-\lambda_0 (u+v+w)}|{\rm J_3}(u,v,w)|\,\Pp_{S_{T-s_2}^B}(\d w)\Pp_{S_{s_2-s_1}^B}(\d v)\Pp_{S_{s_1}^B}(\d u)\\
&\leq c_3r^{-\gamma} \e^{(\lambda_m-\lambda_0)r}\e^{[B(\lambda_m)-B(\lambda_0)]s_1}\e^{-B(\lambda_m) s_2}\e^{-B(\lambda_1)(T-s_2)}\big[(T-s_2)^{-\delta}+1\big],
\end{split}\end{equation*}
for some constant $c_3>0$. This together with $J_1=J_2=0$,  we deduce similarly as \eqref{10L3.2+} that, there exists a constant $c_4>0$ such that
\begin{equation*}\begin{split}
\e^{B(\ll_0)T}\int_0^t \d s_1\int_{s_1}^t\Xi_T(s_1,s_2)\,\d s_2\leq\frac{c_4 \e^{(\lambda_m-\lambda_0)r}r^{-\gamma}}{B(\lambda_m)-B(\lambda_0)},\quad T\geq t\geq 1,\,r\in(0,1),\,m\in\mathbb{N}.
\end{split}\end{equation*}
Plugging this estimate into \eqref{8L3.2} for $\nu=\nu_0$ and every $\gamma>\frac{d+2-6\alpha}{6}$, we obtain  a constant $c_5>0$ such that, for every $t\geq 1$ and every $r\in(0,1)$,
\begin{equation}\begin{split}\label{2-1}
&\sup_{T\geq t}t\E^\nu\big[\mu_0\big(|\nabla (-\mathcal{L}_0)^{-1}(\rho_{t,r}^B-1)|^2\big)|T<\sigma_\tau^B\big]\\
&\leq c_5\sum_{m=1}^\infty\frac{\e^{-2(\lambda_m-\lambda_0)r}}{(\lambda_m-\lambda_0)[B(\lambda_m)-B(\lambda_0)]}+\frac{c_5 r^{-\gamma}}{t}\sum_{m=1}^\infty\frac{\e^{-(\lambda_m-\lambda_0)r}}{(\lambda_m-\lambda_0)[B(\lambda_m)-B(\lambda_0)]},
\end{split}\end{equation}
from which we prove \eqref{0L3.2}.

Now let $d\ge 2(1+\aa)$. Then \eqref{equ-B}, \eqref{BR} and \eqref{EIG} imply that
\begin{equation}\begin{split}\label{2-2}
&\sum_{m=1}^\infty\frac{\e^{-(\lambda_m-\lambda_0)r}}{(\lambda_m-\lambda_0)(B(\lambda_m)-B(\lambda_0))}\leq c_6\sum_{m=1}^\infty\ff{\e^{-(\ll_m-\ll_0)r}}{(\ll_m-\ll_0)^{1+\alpha}}\\
&\leq c_7\int_1^\infty s^{-\ff{2(1+\aa)}{d}}\e^{-c_8 rs^{2/d}}\,\d r\le c_9\Big(r^{-\ff{d-2(1+\aa)}{2}}+1_{\{d=2(1+\aa)\}}\log r^{-1}\Big),\quad r\in(0,1),
\end{split}\end{equation}
for some constants $c_6,c_7,c_8,c_9>0$. Gathering \eqref{2-1} and \eqref{2-2} together, we get the desired estimate \eqref{0L3.3}.

Therefore, we complete the proof.
\end{proof}

The next lemma is useful, which brings us convenience in the sequel.
\begin{lem}\label{L3.3}
Assume that $B\in\mathbf{B}$. For every $t>0$, let $\xi$ be a nonnegative random variable such that $\xi\in\si(X_s^B:\ s\leq t)$, the $\sigma$-algebra generated by $(X_s^B)_{s\in[0,t]}$.  Then, there exists a constant $c>0$ such that
$$\sup_{T\ge t}\E^\nu[\xi|T<\si_\tau^B]\leq c\E^\nu[\xi|t<\si_\tau^B],\quad t\ge 1,\,\nu\in\scr{P}_0.$$
\end{lem}
\begin{proof} By the Markov property, \eqref{DPQ} with $p=q=\infty$, \eqref{SDSG0} and \eqref{LT}, there exists a constant $c_1>0$ such that
\begin{equation*}\begin{split}
\E^\nu[\xi1_{\{T<\si_\tau^B\}}]&=\E^\nu\big[\xi1_{\{t<\si_\tau^B\}}\E^{X_t^B}1_{\{T-t<\si_\tau^B\}}\big]\\
&=\E^\nu\big[\xi1_{\{t<\si_\tau^B\}}P_{T-t}^{D,B}1(X_t^B)\big]\\
&\le \E^\nu[\xi 1_{\{t<\si_\tau^B\}}]\|P_{T-t}^{D,B}\|_{L^\infty(\mu)\rightarrow L^\infty(\mu)}\\
&\le c_1\e^{-B(\ll_0)(T-t)}\E^\nu[\xi1_{\{t<\si_\tau^B\}}],\quad T\geq t>0.
\end{split}\end{equation*}
Next, from \eqref{SDHT}, we see that
$$\Pp^\nu(T<\si_\tau^B)\ge c_2\Pp^\nu(t<\si_\tau^B)\e^{-B(\ll_0)(T-t)},\quad T\geq t\geq1,$$
for some constant $c_2>0$.
Thus, by the definition of the conditional expectation, we have
\begin{equation*}\begin{split}
&\E^\nu[\xi|T<\si_\tau^B]=\frac{\E^\nu[\xi1_{\{T<\si_\tau^B\}}]}{\Pp^\nu(T<\si_\tau^B)}\\
&\le\ff{c_1\E^\nu[\xi 1_{\{t<\si_\tau^B\}}]}{c_2\Pp^\nu(t<\si_\tau^B)}=\ff{c_1}{c_2}\E^\nu[\xi|t<\si_\tau^B],\quad T\geq t\geq1.
\end{split}\end{equation*}
This implies the desired assertion.
\end{proof}

\begin{lem}\label{L3.5}
Let $B\in \BB^\aa$ for some $\aa\in(\frac{1}{2},1]$ and let $d<6\alpha-2$. Then there exists a constant $c>0$ such that, for any $p\in[1,2]$,
$$\sup_{T\geq t}\E^{\nu_0}[|\psi_m^B(t)|^{2p}|t<\sigma_\tau^B]\leq c  m^{\varsigma}t^{-p},\quad t\geq 1,\,m\in\mathbb{N}, $$
where $\varsigma:=\frac{(d+2-2\alpha)p-d-2}{d}$.
\end{lem}
\begin{proof}
We only need to prove the case when $T=t$ instead of for all $T\geq t$ due to Lemma \ref{L3.3}. Applying H\"{o}lder's inequality, we obtain
\begin{equation*}\begin{split}
&\E^{\nu_0}[|\psi_m^B(t)|^{2p}|t<\sigma_\tau^B]=\E^{\nu_0}[|\psi_m^B(t)|^{4-2p}|\psi_m^B(t)|^{4p-4}|t<\sigma_\tau^B]\\
&\leq \big\{\E^{\nu_0}[|\psi_m^B(t)|^2|t<\sigma_\tau^B]\big\}^{2-p}\big\{\E^{\nu_0}[|\psi_m^B(t)|^4|t<\sigma_\tau^B]\big\}^{p-1},\quad t>0.
\end{split}\end{equation*}
According to \eqref{SDHT}, it is enough to find a constant $c>0$ such that the following two inequalities hold, i.e.,
\begin{equation}\label{1L3.5}
\E^{\nu_0}[|\psi_m^B(t)|^2 1_{\{t<\sigma_\tau^B\}}]\leq cm^{-\frac{2\alpha}{d}}\e^{-B(\lambda_0) t}t^{-1},\quad t\geq 1,\,m\in\mathbb{N}, 
\end{equation}
and
\begin{equation}\label{2L3.5}
\E^{\nu_0}[|\psi_m^B(t)|^4 1_{\{t<\sigma_\tau^B\}}]\leq c m^{\frac{d+2-4\alpha}{d}}\e^{-B(\lambda_0) t}t^{-2},\quad t\geq 1,\,m\in\mathbb{N}. 
\end{equation}

Next, we prove \eqref{1L3.5} and \eqref{2L3.5} separately. For convenience, for each $m\in\mathbb{N}$, we denote $\phi_m\phi_0^{-1}$ by $\phi^m_0$ in the following proof.

\textbf{(i)} By Fubini's theorem, we have
\begin{equation}\label{3L3.5}
\E^{\nu_0}[|\psi_m^B(t)|^2 1_{\{t<\sigma_\tau^B\}}]=\frac{2}{t^2}\int_0^t \d s_1\int_{s_1}^t \E^{\nu_0}[1_{\{t<\sigma_\tau^B\}}\phi^m_0(X_{s_1}^B)\phi^m_0(X_{s_2}^B)]\,\d s_2,\quad t>0.
\end{equation}
Similar as \eqref{3L3.2}, we obtain
\begin{equation*}\begin{split}\label{4L3.5}
&\E^{\nu_0}[1_{\{t<\sigma_\tau^B\}}\phi^m_0(X_{s_1}^B)\phi^m_0(X_{s_2}^B)]\\
&=\int_0^\infty\!\!\!\int_0^\infty\!\!\!\int_0^\infty \e^{-\lambda_0 (u+v+w)}\nu_0\big(\phi_0 P_u^0[\phi^m_0P_v^0(\phi^m_0 P_w^0\phi_0^{-1})]\big)\,\Pp_{S_{t-s_2}^B}(\d w)\Pp_{S_{s_2-s_1}^B}(\d v)\Pp_{S_{s_1}^B}(\d u).
\end{split}\end{equation*}
Noting that $\mu_0(|\phi^m_0|^2)=1$ for every $m\in\mathbb{N}$, by \eqref{EIG0}, \eqref{PQ0}, \eqref{PHI}, the symmetry and the invariance of $(P_t^0)_{t\geq0}$ w.r.t. $\mu_0$, we find a constant $c_1>0$ such that, for every $p\in[1,3)$,
\begin{equation*}\begin{split}
&\nu_0\big(\phi_0 P_u^0[\phi^m_0P_v^0(\phi^m_0 P_w^0\phi_0^{-1})]\big)=\frac{1}{\mu(\phi_0)}\mu_0\big(\phi^m_0P_v^0(\phi^m_0 P_w^0\phi_0^{-1})\big)\\
&=\frac{\e^{-(\lambda_m-\lambda_0)v}}{\mu(\phi_0)}\mu_0(|\phi^m_0|^2 P_w^0\phi_0^{-1})
=\frac{\e^{-(\lambda_m-\lambda_0)v}}{\mu(\phi_0)}\big[\mu_0\big(|\phi^m_0|^2 (P_w^0-\mu_0)\phi_0^{-1}\big) +\mu_0(\phi_0^{-1})\big]\\
&\leq\frac{\e^{-(\lambda_m-\lambda_0)v}}{\mu(\phi_0)}\big[\mu_0(|\phi^m_0|^2) \|P_w^0-\mu_0\|_{L^p(\mu_0)\rightarrow L^\infty(\mu_0)}\|\phi_0^{-1}\|_{L^p(\mu_0)} +\mu_0(\phi_0^{-1})\big]\\
&\leq c_1 \e^{-(\lambda_m-\lambda_0)v} \left[(1\wedge w)^{-\frac{d+2}{2p}}\e^{-(\lambda_1-\lambda_0)w}+1\right], \quad u,v,w>0.
\end{split}\end{equation*}
Hence
\begin{equation*}\begin{split}\label{4L3.5+}
&\E^{\nu_0}[1_{\{t<\sigma_\tau^B\}}\phi^m_0(X_{s_1}^B)\phi^m_0(X_{s_2}^B)]\\
&\leq c_1\int_0^\infty\!\!\!\int_0^\infty\!\!\!\int_0^\infty \e^{-\lambda_0u} \e^{-\lambda_m v}\e^{-\lambda_1 w} (1\wedge w)^{-\frac{d+2}{2p}} \,\Pp_{S_{t-s_2}^B}(\d w)\Pp_{S_{s_2-s_1}^B}(\d v)\Pp_{S_{s_1}^B}(\d u)\\
&\quad+c_1\int_0^\infty\!\!\!\int_0^\infty\!\!\!\int_0^\infty \e^{-\ll_0 u}\e^{-\ll_m v}\e^{-\ll_0 w}\,\Pp_{S_{t-s_2}^B}(\d w)\Pp_{S_{s_2-s_1}^B}(\d v)\Pp_{S_{s_1}^B}(\d u)\\
&\leq c_2\e^{-B(\lambda_0)t}\e^{[B(\lambda_m)-B(\lambda_0)]s_1}\e^{-[B(\lambda_m)-B(\lambda_0)]s_2}\Big((t-s_2)^{-\frac{d+2}{2\alpha p}}+1\Big),\quad 0<s_1<s_2<t,
\end{split}\end{equation*}
for some constant $c_2>0$. Substituting this estimate into \eqref{3L3.5} and taking some $p\in(1,3)$ such that $\frac{d+2}{2\alpha p}\in(0,1)$, by \eqref{equ-B}, \eqref{BR} and \eqref{EIG}, we have some constants $c_3,c_4>0$ such that
 $$\E^{\nu_0}[|\psi_m^B(t)|^2 1_{\{t<\sigma_\tau^B\}}]\leq c_3 \frac{\e^{-B(\lambda_0)t}}{[B(\lambda_m)-B(\lambda_0)]t}\leq c_4 \frac{\e^{-B(\lambda_0)t}}{m^{2\alpha/d}t},\quad t\geq1,\,m\in\mathbb{N},$$
 which shows \eqref{1L3.5}.

\textbf{(ii)} For any $s>0$, by Fubini's theorem and the Markov property, we have
\begin{equation}\begin{split}\label{5L3.5}
&s^4\E^{\nu_0}[|\psi_m^B(s)|^4 1_{\{s<\sigma_\tau^B\}}]\\
&=24\int_0^s \d s_1\int_{s_1}^s \d s_2\int_{s_2}^s \d s_3\int_{s_3}^s\E^{\nu_0}[1_{\{s<\sigma_\tau^B\}}
\phi_0^m(X_{s_1}^B)\phi_0^m(X_{s_2}^B)\phi_0^m(X_{s_3}^B)\phi_0^m(X_{s_4}^B)]\,\d s_4\\
&=24\int_0^s \d s_1\int_{s_1}^s \d s_2\int_{s_2}^s \d s_3\int_{s_3}^s\E^{\nu_0}[1_{\{s<\sigma_\tau^B\}}\phi_0^m(X_{s_1}^B)\phi_0^m(X_{s_2}^B)\eta_s(s_3,s_4)]\,\d s_4,
\end{split}\end{equation}
where
$$\eta_s(s_3,s_4):=\E^{\nu_0}[1_{\{s<\sigma_\tau^B\}}\phi_0^m(X_{s_3}^B)\phi_0^m(X_{s_4}^B)|\sigma(X_r^B:\ r\leq s_3)],\quad 0<s_3<s_4<s.$$
By the Markov property
and \eqref{SDSG0}, for every $0<s_3<s_4<s$, $\eta_s(s_3,s_4)$ can be written  as
\begin{equation}\begin{split}\label{6L3.5}
\eta_s(s_3,s_4)&=\phi_0^m(X_{s_3}^B)\E^{X_{s_3}^B}[1_{\{s-s_3<\sigma_\tau^B\}}\phi_0^m(X_{s_4-s_3}^B)]\\
&=\int_0^\infty\!\!\!\int_0^\infty \e^{-\lambda_0 u}\e^{-\lambda_0 v}\{\phi_0\phi_0^m P_v^0(\phi_0^m P_u^0\phi_0^{-1})\}(X_{s_3}^B)\,\Pp_{S_{s-s_4}^B}(\d u)\Pp_{S_{s_4-s_3}^B}(\d v).
\end{split}\end{equation}
According to Fubini's theorem, the Markov property, \eqref{5L3.5} and \eqref{6L3.5}, we have
\begin{equation}\begin{split}\label{7L3.5}
I(s):&=s^4 \e^{B(\ll_0) s}\E^{\nu_0}[|\psi_m^B(s)|^4 1_{\{s<\sigma_\tau^B\}}]\\
&=12 \e^{ B(\ll_0) s}\int_0^s \d r_1\int_{r_1}^s \E^{\nu_0}\left[1_{\{r_1<\si_\tau^B\}}\eta_s(r_1,r_2)\left|\int_0^{r_1}\phi_0^m(X_r^B)\d r\right|^2\right]\,\d r_2\\
&=12\int_0^s \d r_1 \int_{r_1}^s \E^{\nu_0}\left[1_{\{r_1<\si_\tau^B\}}\e^{B(\ll_0) s-\ff{B(\ll_0) r_1} 2}\eta_s(r_1,r_2)\,\e^{\ff{B(\ll_0) r_1} 2}\left|\int_0^{r_1}\phi_0^m(X_r^B)\d r\right|^2\right]\,\d r_2\\
&\le12\int_0^s \d r_1\int_{r_1}^s\big(\E^{\nu_0}[1_{\{r_1<\si_\tau^B\}}\e^{2B(\ll_0)s-B(\ll_0) r_1}\eta_s(r_1,r_2)^2]\big)^{\ff 1 2}\\
&\quad\times \left(\E^{\nu_0}\left[\e^{B(\ll_0) r_1}\left|\int_0^{r_1}\phi_0^m(X_r^B)\d r\right|^4 1_{\{r_1<\si_\tau^B\}}\right]\right)^{\ff 1 2}\,\d r_2\\
&=12\int_0^s \d r_1\int_{r_1}^s\big(\E^{\nu_0}[1_{\{r_1<\si_\tau^B\}}\e^{2B(\ll_0)s-B(\ll_0) r_1}\eta_s(r_1,r_2)^2]\big)^{\ff 1 2}\\
&\quad\times \big(\E^{\nu_0}[\e^{B(\ll_0) r_1}r_1^4 |\psi_m^B(r_1)|^4 1_{\{r_1<\si_\tau^B\}}]\big)^{\ff 1 2}\,\d r_2\\
&\le 12\Big(\sup_{r\in [0,s]}\sqrt{I(r)}\Big)\int_0^s \d r_1 \int_{r_1}^s\big\{\e^{2B(\ll_0) s-B(\ll_0) r_1}\E^{\nu_0}[1_{\{r_1<\si_\tau^B\}}|\eta_s(r_1,r_2)|^2]\big\}^{\ff 1 2}\,\d r_2,\quad s>0,
\end{split}\end{equation}
which clearly implies that, for every $t>0$,
\begin{equation}\begin{split}\label{8L3.5}
I(t)\le \sup_{s\in[0,t]}I(t)\le\left(12\sup_{s\in[0,t]}\int_0^s \d r_1\int_{r_1}^s\left\{\e^{B(\ll_0)(2 s- r_1)}\E^{\nu_0}[1_{\{r_1<\si_\tau^B\}}|\eta_s(r_1,r_2)|^2]\right\}^{\ff 1 2}\,\d r_2\right)^2.
\end{split}\end{equation}
So, it remains to estimate the upper bound of the rightmost term in \eqref{8L3.5}.

For convenience, let
$$g(x)=\int_0^\infty\!\!\!\int_0^\infty \e^{-\ll_0 (u+v)}\{\phi_0^m\phi_0 P_v^0(\phi_0^m P_u^0\phi_0^{-1})\}(x)\,\P_{S_{s-r_2}^B}(\d u)\P_{S_{r_2-r_1}^B}(\d v),\quad x\in M.$$
Then, by \eqref{6L3.5}, \eqref{SDSG0}, Fubini's theorem and the invariance of $(P_t)_{t\geq0}$ w.r.t. $\mu_0$, we derive that
\begin{equation*}\begin{split}\label{9L3.5}
&\E^{\nu_0}[1_{\{r_1<\si_\tau^B\}}|\eta_s(r_1,r_2)|^2]=\E^{\nu_0}[1_{\{r_1<\si_\tau^B\}}|g(X_{r_1}^B)|^2]\\
&=\ff 1 {\mu(\phi_0)}\int_M \phi_0^{-1}(x)P_{r_1}^{D,B}(|g|^2)(x)\,\mu_0(\d x)\\
&=\ff 1 {\mu(\phi_0)}\int_M \int_0^\infty \e^{-\lambda_0 r} P_r^0(|g|^2\phi_0^{-1})(x)\,\P_{S_{r_1}^B}(\d r)\mu_0(\d x)\\
&=\ff 1 {\mu(\phi_0)}\int_0^\infty \e^{-\lambda_0 r} \mu_0(|g|^2\phi_0^{-1})\,\P_{S_{r_1}^B}(\d r),
\end{split}\end{equation*}
where
\begin{equation*}\begin{split}\label{9L3.5+}
&\mu_0(|g|^2\phi_0^{-1})=\mu_0\left(\phi_0\Big|\int_0^\infty\!\!\!\int_0^\infty \e^{-\ll_0 (u+v)} \phi_0^m P_v^0(\phi_0^m P_u^0\phi_0^{-1}) \,\P_{S_{s-r_2}^B}(\d u)\P_{S_{r_2-r_1}^B}(\d v)\Big|^2\right)\\
&\leq 2 \mu(\phi_0)^2\mu_0\left(\phi_0\Big|\int_0^\infty\!\!\!\int_0^\infty \e^{-\ll_0 (u+v)} \phi_0^mP_v^0\phi_0^m\,\P_{S_{s-r_2}^B}(\d u)\P_{S_{r_2-r_1}^B}(\d v)\Big|^2\right)\\
&\quad +2\mu_0\left(\phi_0\Big|\int_0^\infty\!\!\!\int_0^\infty \e^{-\ll_0 (u+v)}  \phi_0^mP_v^0(\phi_0^m [P_u^0-\mu_0]\phi_0^{-1}) \,\P_{S_{s-r_2}^B}(\d u)\P_{S_{r_2-r_1}^B}(\d v)\Big|^2\right)\\
&=: {\rm I_1}+{\rm I_2}.
\end{split}\end{equation*}
By \eqref{EIG0}, \eqref{LT} and $\mu_0(|\phi_0^m|^2)=1$ for each $m\in\mathbb{N}$, it is easy to see that
\begin{equation*}\begin{split}\label{I-1}
 {\rm I_1}&=2 \mu(\phi_0)^2\mu_0(\phi_0|\phi_0^m|^4)\e^{-2B(\lambda_0)(s-r_2)}\e^{-2B(\lambda_m)(r_2-r_1)}\\
 &\leq2\mu(\phi_0)^2\|\phi_0\|_\infty\|\phi_0^m\|_\infty^2\e^{-2B(\lambda_0)(s-r_2)}\e^{-2B(\lambda_m)(r_2-r_1)},
\end{split}\end{equation*}
for any $m\in\mathbb{N}$ and any $0<r_1<r_2<s$. By the symmetry of $(P_t^0)_{t>0}$, \eqref{EIG0}, and $\|\phi_0^m\|_{L^2(\mu_0)}=1$ for each $m\in\mathbb{N}$, we obtain
\begin{equation*}\begin{split}\label{12L3.5}
&\|\phi_0^mP_v^0(\phi_0^m [P_u^0-\mu_0]\phi_0^{-1})\|_{L^2(\mu_0)}\\
&=\e^{-(\lambda_m-\lambda_0)v} \||\phi_0^m|^2 [P_u^0-\mu_0]\phi_0^{-1}\|_{L^2(\mu_0)}\\
&\leq\e^{-(\lambda_m-\lambda_0)v}\|\phi_0^m\|_\infty\|[P_u^0-\mu_0]\phi_0^{-1}\|_{L^\infty(\mu_0)} \|\phi_0^m\|_{L^2(\mu_0)}\\
&\leq\e^{-(\lambda_m-\lambda_0)v}\|\phi_0^m\|_\infty \|P_u^0-\mu_0\|_{L^p{(\mu_0)}\to L^\infty(\mu_0)}\|\phi_0^{-1}\|_{L^p(\mu_0)},\quad m\in\mathbb{N},\,u,v>0,
\end{split}\end{equation*}
for every $p\in(\ff{d+2}{2\aa},3)$.
Then, by Minkowski's inequality, \eqref{PHI}, \eqref{PQ0} and \eqref{LT}, for any $p\in(\ff{d+2}{2\aa},3)$, we find a constant $c_5>0$ such that,  for any $m\in\mathbb{N}$ and any $0<r_1<r_2<s$,
\begin{equation*}\begin{split}\label{12L3.5}
\sqrt{\rm I_2}
&\leq \sqrt{2}\|\phi_0\|_\infty^{\ff 1 2}
\int_0^\infty\!\!\!\int_0^\infty \e^{-\ll_0 (u+v)}  \|\phi_0^mP_v^0(\phi_0^m [P_u^0-\mu_0]\phi_0^{-1})\|_{L^2(\mu_0)} \,\P_{S_{s-r_2}^B}(\d u)\P_{S_{r_2-r_1}^B}(\d v)\\
&\leq c_5\|\phi_0^m\|_\infty
\int_0^\infty\!\!\!\int_0^\infty \e^{-\ll_m v}\e^{-\ll_1 u}(1\wedge u)^{-\frac{d+2}{2p}}
\,\P_{S_{s-r_2}^B}(\d u)\P_{S_{r_2-r_1}^B}(\d v)\\
&=c_5\|\phi_0^m\|_\infty\e^{-B(\lambda_m)(r_2-r_1)}\e^{-B(\lambda_1)(s-r_2)}\Big[1+(s-r_2)^{-\frac{d+2}{2p\alpha}}\Big],
\end{split}\end{equation*}
or
\begin{equation*}\begin{split}\label{I-2}
 {\rm I_2}\leq c_5^2\|\phi_0^m\|_\infty^2\e^{-2B(\lambda_m)(r_2-r_1)}\e^{-2B(\lambda_0)(s-r_2)}\Big[1+(s-r_2)^{-\frac{d+2}{2p\alpha}}\Big]^2,
 \quad m\in\mathbb{N},\,0<r_1<r_2<s.
\end{split}\end{equation*}
Thus, by \eqref{EIG0UB} and \eqref{LT}, for any $p\in(\ff{d+2}{2\aa},3)$, we have a constant $c_6>0$ such that, for every $0<r_1<r_2<s$ and each $m\in\mathbb{N}$,
\begin{equation}\begin{split}\label{I}
&\E^{\nu_0}[1_{\{r_1<\si_\tau^B\}}|\eta_s(r_1,r_2)|^2]\leq\frac{1}{\mu(\phi_0)}\int_0^\infty \e^{-\lambda_0 r}({\rm I_1}+{\rm I_2})\,\P_{S_{r_1}^B}(\d r)\\
&\leq c_6\|\phi_0^m\|_\infty^2\e^{-B(\ll_0)(2s-r_1)}\e^{-2[B(\ll_m)-B(\ll_0)](r_2-r_1)}\Big[1+(s-r_2)^{-\frac{d+2}{2p\alpha}}\Big]^2.
\end{split}\end{equation}

According to \eqref{8L3.5}, \eqref{I}, \eqref{BR}, \eqref{equ-B}, \eqref{EIG0UB} and \eqref{EIG}, we arrive at
$$I(t)\leq c_7t^2 m^{\frac{d+2-4\alpha}{d}},\quad m\in\mathbb{N},\,t\geq1,$$
for some constant $c_7>0$. By the definition of $I(t)$ in  \eqref{7L3.5}, we prove \eqref{2L3.5}.

Therefore, the proof is finished.
\end{proof}

\begin{lem}\label{L3.6}
Let $B\in \BB^\aa$ for some $\aa\in(\ff 1 2,1]$ and $d<6\aa-2$. Then, for every $p\in(1,\frac{2d+6}{3d+4-2\alpha}\wedge \ff {d+2}{d+1})$, there exists a constant $c>0$ such that
$$\sup_{r>0,\,T\ge t}\E^{\nu_0}[\mu_0(|\nabla (-\mathcal{L}_0)^{-1}(\rho^B_{t,r}-1)|^{2p})|T<\si_\tau^B]\le c t^{-p},\quad t\ge 1.$$
\end{lem}
\begin{proof}
It is enough to prove the case when $T=t$ instead of for all $T\geq t$ due to Lemma \ref{L3.3}. Let $p\in(1,\frac{2d+6}{3d+4-2\alpha}\wedge \ff {d+2}{d+1})$. Then, it is easy to see that
$$\ff{(d +2)(2p-2)}4+\ff{d(p-1)} 2+\frac{[(d+2-2\alpha)p-2]^+}{2}<1.$$
Due to this, there exists $\vv\in(0,1)$ such that
\begin{equation}\label{2L3.6}
\kappa:=\ff{(d+2)(2p-2+\vv)}4+\ff{d(p-1)} 2+\frac{[(d+2-2\alpha)p-2]^+}{2}<1.
\end{equation}

Recall \cite[(2.21)]{eW2}, i.e., for any $\varsigma>0$ and any $q>1$, there exists a constant $c>0$ such that
$$|\nn P_t^0 f|\le \ff{c\phi_0^{-\varsigma}}{\sqrt{1\wedge t}}(P_t^0|f|^q)^{\ff 1 q},\quad t>0,\,f\in\mathscr{B}_b(\mathring{M}).$$
Then, combining this with $(-\mathcal{L}_0)^{-1}=\int_0^\infty P_s^0\,\d s$, $\mu_0(\rho_{t,r}^B-1)=0$ and H\"{o}lder's inequality, there exists a constant $c_1>0$ such that
\begin{equation}\begin{split}\label{3L3.6}
&\mu_0\big(|\nabla (-\mathcal{L}_0)^{-1}(\rho_{t,r}^B-1)|^{2p}\big)\le\int_M\left(\int_0^\infty|\nabla P_s^0(\rho_{t,r}^B-1)|\d s\right)^{2p}\,\d \mu_0\\
&\le c_1\int_M\left(\int_0^\infty\ff{1}{\sqrt{1\wedge s}}\big[P_{\ff s 4}^0|P_{\ff{3s} 4}^0(\rho_{t,r}^B-1)|^p\big]^{\ff 1 p}\d s\right)^{2p}\phi_0^{-\vv}\,\d \mu_0\\
&\le c_1 \int_0^\infty \e^{2p\theta s}\mu_0\big(\phi_0^{-\vv} [P_{\ff s 4}^0|P_{\ff{3s} 4}^0(\rho_{t,r}^B-1)|^p ]^2\big)\,\d s\\
&\quad\times \left(\int_0^\infty(1\wedge s)^{-\ff p{2p-1}}\e^{-\ff{2p\theta s}{2p-1}}\,\d s\right)^{\ff{2p-1}{2p}},\quad t,r,\theta>0.
\end{split}\end{equation}

Since $p>1$, it is clear that
\begin{equation}\begin{split}\label{4L3.6}
\int_0^\infty(1\wedge s)^{-\ff{p}{2p-1}}\e^{-\ff{2p\theta s}{2p-1}}\,\d s<\infty,\quad \theta>0.
\end{split}\end{equation}
As for the estimation of $\mu_0\big(\phi_0^{-\vv} [P_{\ff s 4}^0|P_{\ff{3s} 4}^0(\rho_{t,r}^B-1)|^p ]^2\big)$,  using the fact that $\|\phi_0^{-\vv}\|_{L^{2/\vv}(\mu_0)}=1$, $\mu_0(\rho_{t,r}^B-1)=0$, and $P_t^0$ is contractive in $L^q(\mu_0)$ for all $q\ge 1$, by \eqref{PQ0} and H\"{o}lder's inequality, we find a constant $c_2>0$ such that
\begin{equation}\begin{split}\label{4L3.6+}
&\mu_0\big(\phi_0^{-\vv} [P_{\ff s 4}^0|P_{\ff{3s} 4}^0(\rho_{t,r}^B-1)|^p ]^2\big)\le\|P_{\ff s 4}^0|P_{\ff{3s} 4}^0(\rho_{t,r}^B-1)|^p\|_{L^{\ff 4{2-\vv}}(\mu_0)}^2\|\phi_0^{-\vv}\|_{L^{2/\vv}(\mu_0)}\\
&\leq \|P_{\ff{3s} 4}^0(\rho_{t,r}^B-1)\|_{L^{\ff {4p}{2-\vv}}(\mu_0)}^{2p}=\|(P_{\ff s 2}^0-\mu_0)[P_{\ff s 4}^0(\rho_{t,r}^B-1)]\|_{L^{\ff{4p}{2-\vv}}(\mu_0)}^{2p}\\
&\le \|P_{\ff s 2}^0-\mu_0\|_{L^2(\mu_0)\to L^{\ff{4p}{2-\vv}}(\mu_0)}^{2p}\|P_{\ff s 4}^0(\rho_{t,r}^B-1)\|_{L^2(\mu_0)}^{2p}\\
&\le c_2(1\wedge s)^{-\ff{(d+2)(2p-2+\vv)}4}\e^{-(\ll_1-\ll_0)ps}\|P_{\ff s 4}^0(\rho_{t,r}^B-1)\|_{L^2(\mu_0)}^{2p},\quad t,r,s>0.
\end{split}\end{equation}
By \eqref{EIG0}, \eqref{MTR}, H\"{o}lder's inequality and the fact that $(\phi_m\phi_0^{-1})_{m\in\mathbb{N}}$ is an orthonormal basis in $L^2(\mu_0)$, we find a constant $c_3>0$ so that
\begin{equation*}\begin{split}\label{4L3.6++}
&\|P_{\ff s 4}^0(\rho_{t,r}^B-1)\|_{L^2(\mu_0)}^{2p}=\left(\sum_{m=1}^\infty \e^{-(\ll_m-\ll_0)(2r+s/2)}|\psi_m^B(t)|^2\right)^p\\
&\le \left(\sum_{m=1}^\infty \e^{-(\ll_m-\ll_0)(2r+s/2)}\right)^{p-1}\sum_{m=1}^\infty \e^{-(\ll_m-\ll_0)(2r+s/2)}|\psi_m^B(t)|^{2p}\\
&\leq c_3(1\ww s)^{-\ff {d(p-1)} 2}\sum_{m=1}^\infty \e^{-(\ll_m-\ll_0)(2r+s/2)}|\psi_m^B(t)|^{2p},\quad t,r,s>0,
\end{split}\end{equation*}
where in the last inequality we used the following estimate, i.e., by \eqref{EIG},
$$\sum_{m=1}^\infty \e^{-(\ll_m-\ll_0)(2r+s/2)}\le \aa_1\int_1^\infty \e^{-\aa_2(2r+s/2)t^{2/d}}\d t\le\aa_3(1\ww s)^{-\ff d 2},\quad r,s>0,$$
for some constants $\aa_1,\aa_2,\aa_3>0$. Hence, by Lemma \ref{L3.5} and \eqref{EIG},
\begin{equation}\begin{split}\label{4L3.6++}
&\E^{\nu_0}[\|P_{\ff s 4}^0(\rho_{t,r}^B-1)\|_{L^2(\mu_0)}^{2p}|t<\sigma_\tau^B]\\
&\leq c_3(1\ww s)^{-\ff{d(p-1)}{2}}\sum_{m=1}^\infty \e^{-(\ll_m-\ll_0)(2r+s/2)}\E^{\nu_0}[|\psi_m^B(t)|^{2p}|t<\sigma_\tau^B]\\
&\leq c_4 t^{-p}(1\ww s)^{-\ff{d(p-1)}{2}}\sum_{m=1}^\infty \e^{-c_5sm^{2/d}}m^{\frac{(d+2-2\alpha)p-d-2}{d}}\\
&\leq c_6 t^{-p}(1\ww s)^{-\ff{d(p-1)}{2}}\int_1^\infty \e^{-c_5 su^{\ff 2 d}}u^{\frac{(d+2-2\alpha)p-d-2}{d}}\,\d u\\
&\leq c_7 t^{-p}(1\ww s)^{-\ff{d(p-1)}{2}-\frac{(d+2-2\alpha)p-2}{2}}\int_s^\infty v^{\frac{(d+2-2\alpha)p-2}{2}-1}\e^{-v} \,\d v\\
&\leq c_8 t^{-p}(1\ww s)^{-\ff{d(p-1)}{2}-\frac{[(d+2-2\alpha)p-2]^+}{2}}\log(2+s^{-1}),\quad t\geq1,\,r,s>0,
\end{split}\end{equation}
for some constants $c_i>0$, $i=4,5,6,7,8$, where the term $\log(2+s^{-1})$ comes from the situation when $(d+2-2\alpha)p=2$.

Putting  \eqref{4L3.6++}, \eqref{4L3.6+}, \eqref{4L3.6} and \eqref{3L3.6} together, we can find a function $a:(0,\infty)\to(0,\infty)$ such that
\begin{equation*}\begin{split}\label{5L3.6}
&\E^{\nu_0}[\mu_0(|\nn (-\mathcal{L}_0)^{-1}(\rho_{t,r}^B-1)|^{2p})|t<\si_\tau^B]\\
&\le a(\tt)t^{-p}\int_0^\infty (1\ww s)^{-\kappa}\log(2+s^{-1}) \e^{2p\tt s-(\ll_1-\ll_0)ps}\,\d s,\quad t\geq1,\,\theta>0,
\end{split}\end{equation*}
where $\kappa$ is defined in \eqref{2L3.6} above. Choosing $\theta\in(0,\ff{\ll_1-\ll_0} 2)$, since $\kappa\in(0,1)$, we finally prove the desired result.
\end{proof}

\begin{lem}\label{L3.4}
Assume that $B\in \BB^\aa$ for some $\aa\in(\frac{1}{2},1]$ and $d<6\alpha-2$. Then for any $\varepsilon\in(\frac{d}{4}\vee\frac{d^2+2(1-\aa)d}{2(d+2)},1)$, there exists a constant $c>0$ such that
$$\sup_{T\ge t}\E^{\nu_0}[|\rho_{t,r}^B(y)-1|^2|T<\si_\tau^B]\leq c\phi_0^{-2}(y) t^{-1}r^{-\vv},\quad t\geq 1,\,r\in(0,1],\,y\in \mathring{M}.$$
\end{lem}
\begin{proof} Let $t,r>0$. Recall that $\nu_0=\frac{\phi_0}{\mu(\phi_0)}\mu$. By Lemma \ref{L3.3}, it is clear that we only need to establish the upper bound for $\E^{\nu_0}[|\rho_{t,r}^B(y)-1|^2|t<\si_\tau^B]$.  For an arbitrary fixed point $y\in \mathring{M}$, let
$$f(\cdot):=p_r^0(\cdot,y)-1.$$
Observing that \eqref{MTR} leads to
$$\rho_{t,r}^B(y)-1=\frac{1}{t}\int_0^t f(X_s^B)\,\d s,$$
we have
\begin{equation}\label{1L3.4}
\E^{\nu_0}[|\rho_{t,r}^B(y)-1|^21_{\{t<\si_\tau^B\}}]=\frac{2}{t^2}\int_0^t \d s_1\int_{s_1}^t\E^{\nu_0}[1_{\{t<\si_\tau^B\}}f(X_{s_1}^B)f(X_{s_2}^B)]\,\d s_2.
\end{equation}
By \eqref{3L3.2} above, we obtain
\begin{equation}\begin{split}\label{2L3.4}
&\E^{\nu_0}[1_{\{t<\si_\tau^B\}}f(X_{s_1}^B)f(X_{s_2}^B)]\\
&=\ff 1 {\mu(\phi_0)}\int_0^\infty\!\!\!\int_0^\infty\!\!\!\int_0^\infty \e^{-\lambda_0 (u+v+w)}\mu_0\big(P_u^0[fP_v^0(fP_w^0\phi_0^{-1})]\big)\,\Pp_{S_{t-s_2}^B}(\d w)\Pp_{S_{s_2-s_1}^B}(\d v)\Pp_{S_{s_1}^B}(\d u),
\end{split}\end{equation}
for every $s_1<s_2<t$.

Noting that $\mu_0(f)=0$, by the invariance and the symmetry of $(P_t^0)_{t\geq0}$ w.r.t. $\mu_0$, we have
\begin{equation*}\begin{split}\label{3L3.4}
&\mu_0\big(P_u^0[fP_v^0(fP_w^0\phi_0^{-1})]\big)=\mu_0\big(fP_v^0(fP_w^0\phi_0^{-1})\big)=\mu_0\big([fP_w^0\phi_0^{-1}]P_v^0 f\big)\\
&=\mu_0\big([fP_w^0\phi_0^{-1}](P_v^0-\mu_0)f\big)=\mu_0\big(f[(P_w^0-\mu_0)\phi_0^{-1}+\mu(\phi_0)](P_v^0-\mu_0)f\big)\\
&=\mu_0\big(f[(P_w^0-\mu_0)\phi_0^{-1}](P_v^0-\mu_0)f\big)+\mu_0\big(f[P_v^0-\mu_0]f\big)\mu(\phi_0)\\
&=:{\rm I}_1+{\rm I}_2,\quad u,v,w>0.
\end{split}\end{equation*}
For any $\alpha\in(\frac{1}{2},1]$ and $d<6\alpha-2$, take $q\in(\frac{d+2}{2\alpha},3)$ so that $\vv_1:=\frac{d+2}{2\alpha q}<1$. Then $\|\phi_0^{-1}\|_{L^q(\mu_0)}<\infty$ by  \eqref{PHI}. For any $p\in(1,2]$, we deduce from \eqref{PQ0} and H\"{o}lder's inequality that, there exist constants $c_1,c_2>0$ such that
\begin{equation*}\begin{split}\label{4L3.4}
{\rm I}_1&\leq\|f\|_{L^p(\mu_0)}\|(P_w^0-\mu_0)\phi_0^{-1}\|_\infty\|(P_v^0-\mu_0)f\|_{L^{\frac{p}{p-1}}(\mu_0)}\\
&\leq \|f\|_{L^p(\mu_0)} \|P_w^0-\mu_0\|_{L^q(\mu_0)\rightarrow L^\infty(\mu_0)}\|\phi_0^{-1}\|_{L^q(\mu_0)}\|P_v^0-\mu_0\|_{L^2(\mu_0)\to L^{\frac{p}{p-1}}(\mu_0)}\|f\|_{L^2(\mu_0)}\\
&\leq c_1 \|f\|_{L^p(\mu_0)}\|f\|_{L^2(\mu_0)}\e^{-(\lambda_1-\lambda_0)v}\{1\wedge v\}^{-\frac{(d+2)(2-p)}{4p}}
\e^{-(\lambda_1-\lambda_0)w}\{1\wedge w\}^{-\frac{d+2}{2q}},\quad v,w>0,
\end{split}\end{equation*}
and similarly,
\begin{equation*}\begin{split}\label{4L3.4}
{\rm I}_2&\leq\mu(\phi_0)\|f\|_{L^p(\mu_0)}\|P_v^0-\mu_0\|_{L^2(\mu_0)\rightarrow L^{\frac{p}{p-1}}(\mu_0)}\|f\|_{L^2(\mu_0)}\\
&\leq c_2\|f\|_{L^p(\mu_0)}\|f\|_{L^2(\mu_0)}\e^{-(\lambda_1-\lambda_0)v}\{1\wedge v\}^{-\frac{(d+2)(2-p)}{4p}},\quad v>0.
\end{split}\end{equation*}
It is well known that $\inf_M\phi_0^{-1}>0$. By \eqref{DPQ} and \eqref{R}, there exist constants $a_1,a_2>0$ such that
\begin{equation*}\begin{split}\label{5L3.4}
\|f\|_{L^p(\mu_0)}&\le 1+\|p_r^0(\cdot,y)\|_{L^p(\mu_0)}= 1+\e^{r\lambda_0}\phi_0^{-1}(y)\|\phi_0^{-1}p_r^D(\cdot,y)\|_{L^p(\mu_0)}\\
&\leq 1+a_1\phi_0^{-1}(y)\|\phi_0\|_\infty^\frac{2-p}{p}\|p_r^D(\cdot,y)\|_{L^p(\mu)}\\
&\leq a_2\phi_0^{-1}(y)r^{-\frac{d(p-1)}{2p}},\quad r\in(0,1],\,p\in[1,2].
\end{split}\end{equation*}
Hence, we derive that there exists a constant $c_3>0$ such that, for every $p\in(1,2]$,
\begin{equation*}\begin{split}\label{5L3.4+}
\mu_0\big(P_u^0[fP_v^0(fP_w^0\phi_0^{-1})]\big)&\leq c_3\phi_0^{-2}(y)r^{-\frac{d(p-1)}{2p}-\frac{d}{4}}
\{1\wedge v\}^{-\frac{(d+2)(2-p)}{4p}}\e^{-(\lambda_1-\lambda_0)v}\\
&\quad\times\Big[1+\{1\wedge w\}^{-\frac{d+2}{2q}}\e^{-(\lambda_1-\lambda_0)w}\Big],\quad u,v,w>0,\,r\in(0,1].
\end{split}\end{equation*}

Thus, by an argument analogous to \eqref{10L3.2}, there exist constants $c_4,c_5>0$ such that, for every $p\in(1,2]$,
\begin{equation}\begin{split}\label{6L3.4}
&\int_0^\infty\!\!\!\int_0^\infty\!\!\!\int_0^\infty \e^{-\lambda_0 (u+v+w)} \mu_0\big(P_u^0[fP_v^0(fP_w^0\phi_0^{-1})]\big)\,\Pp_{S_{t-s_2}^B}(\d w)\Pp_{S_{s_2-s_1}^B}(\d v)\Pp_{S_{s_1}^B}(\d u)\\
&\leq c_4\phi_0^{-2}(y)r^{-\frac{d(p-1)}{2p}-\frac{d}{4}}\e^{-B(\lambda_0) s_1}\e^{-B(\lambda_1)(s_2-s_1)}\Big[1+(s_2-s_1)^{-\frac{(d+2)(2-p)}{4\alpha p}}\Big]\\
&\quad\times\left\{\e^{-B(\lambda_0)(t-s_2)}+\e^{-B(\lambda_1)(t-s_2)}\Big[1+(t-s_2)^{-\varepsilon_1}\Big]\right\}\\
&\leq c_5\phi_0^{-2}(y)r^{-\frac{d(p-1)}{2p}-\frac{d}{4}}\e^{-B(\lambda_0) s_1}\e^{-B(\lambda_1)(s_2-s_1)}\e^{-B(\lambda_0)(t-s_2)}\\
&\quad\times\Big[1+(t-s_2)^{-\varepsilon_1}\Big]\Big[1+(s_2-s_1)^{-\frac{(d+2)(2-p)}{4\alpha p}}\Big],\quad r\in(0,1],\, s_1<s_2<t.
\end{split}\end{equation}
Letting $$p_0:=1\vee\frac{2(d+2)}{d+2+4\alpha},$$ and taking $p>p_0$ such that
$$\varepsilon_2:=\frac{(d+2)(2-p)}{4\alpha p}<1,$$
we can deduce from \eqref{6L3.4} that
\begin{equation}\begin{split}\label{7L3.4}
&\int_0^\infty\!\!\!\int_0^\infty\!\!\!\int_0^\infty \e^{-\lambda_0 (u+v+w)} \mu_0\big(P_u^0[fP_v^0(fP_w^0\phi_0^{-1})]\big)\,\Pp_{S_{t-s_2}^B}(\d w)\Pp_{S_{s_2-s_1}^B}(\d v)\Pp_{S_{s_1}^B}(\d u)\\
&\leq c_6\phi_0^{-2}(y)r^{-\frac{d(p-1)}{2p}-\frac{d}{4}}
\e^{-B(\lambda_0) s_1}\e^{-B(\lambda_1)(s_2-s_1)} \e^{-B(\lambda_0)(t-s_2)}\\
&\quad\times\big[1+(t-s_2)^{-\varepsilon_1}\big]\big[1+(s_2-s_1)^{-\varepsilon_2}\big],\quad r\in(0,1],\,s_1<s_2<t,
\end{split}\end{equation}
for some constant $c_6>0$.

Combing \eqref{7L3.4} with \eqref{2L3.4} and \eqref{1L3.4}, we arrive at
$$\E^{\nu_0}[|\rho_{t,r}^B(y)-1|^2|t<\sigma_\tau^B]\leq c_7\phi_0^{-2}(y)t^{-1}r^{-\frac{d(p-1)}{2p}-\frac{d}{4}},\quad r\in(0,1],\, t\geq 1,$$
for some constant $c_7>0$. Note that $p\mapsto \frac{d(p-1)}{2p}$ is increasing, and for every $\alpha\in(\frac{1}{2},1]$ and every $d<6\alpha-2$,
$$\lim_{p\downarrow p_0}\left\{\frac{d(p-1)}{2p}+\frac{d}{4}\right\}=\frac{d}{4}\vee\frac{d^2+2(1-\alpha)d}{2(d+2)}<1.$$
Hence, for any $\varepsilon\in(\frac{d}{4}\vee\frac{d^2+2(1-\alpha)d}{2(d+2)},1)$, there exists some $p>p_0$ such that $\frac{d(p-1)}{2p}+\frac{d}{4}\leq \varepsilon$.
Therefore, we complete the proof.
\end{proof}

For every $t>0$ and $r\in(0,1]$, let
$$\rho_{t,r,r}^B:=(1-r)\rho_{t,r}^B+r,\quad \mu_{t,r,r}^B:=\rho_{t,r,r}^B\mu_0.$$
\begin{lem}\label{L3.7}
Let $B\in \BB^\aa$ for some $\aa\in(\ff 1 2,1]$ and $d<6\aa-2$. If $r_t=t^{-\zeta}$ for some $\zeta\in(1,\ff 4 d \wedge\ff{2(d+2)}{d^2+2(1-\aa)d})$, then
$$\lim_{t\to\infty}\sup_{T\ge t}\E^{\nu_0}\big[\mu_0\big(|\mathscr{M}(\rho_{t,r_t,r_t}^B,1)^{-1}-1|^q\big)|T<\si_\tau^B\big]=0,\quad t\geq1,\, q\ge 1.$$
\end{lem}
\begin{proof}
By Lemma \ref{L3.3}, instead of $T\ge t$, it suffices to prove the case when $T=t$. By the proof of \cite[Lemma 3.2(2)]{eWW} (see also \cite[(3.16)]{eWZ}), we have, for every $\eta\in(0,1)$ and every $x\in M$,
\begin{equation*}\begin{split}
&\E^{\nu_0}\big[|\mathscr{M}\big(\rho_{t,r_t,r_t}^B(x),1\big)^{-1}-1|^q|t<\si_\tau^B\big]\\
&\le\left|\ff 1 {\sqrt{1-\eta}}-\ff 2{2+\eta}\right|^q
+\Pp^{\nu_0}\big(|\rho_{t,r_t}^B(x)-1|>\eta|t<\si_\tau^B\big),\quad t\geq1.
\end{split}\end{equation*}
According to Lemma \ref{L3.4}, there exist constants $c>0$ and $\vv\in(0,\zeta^{-1})$ such that
$$\Pp^{\nu_0}\big(|\rho_{t,r_t}^B(x)-1|>\eta|t<\si_\tau^B\big)\le c\eta^{-2}\phi_0(x)^{-2}t^{-1+\zeta\vv},\quad x\in M,\,t\geq1.$$
Noting also that $\mu_0(\phi_0^{-2})=1$, we obtain
\begin{equation*}\begin{split}
&\E^{\nu_0}\big[\mu_0\big(|\mathscr{M}\big(\rho_{t,r_t,r_t}^B,1\big)^{-1}-1|^q\big)|t<\si_\tau^B\big]\\
&\le\left|\ff 1 {\sqrt{1-\eta}}-\ff 2{2+\eta}\right|^q+c\eta^{-2}t^{-1+\zeta\vv},\quad \eta\in(0,1),\,t\ge 1.
\end{split}\end{equation*}
Since $\zeta\vv\in(0,1)$, by letting first $t\to\infty$ and then $\eta\to 0^+$, we finish the proof.
\end{proof}

Let ${\rm diam}(M)$ denote the diameter of $M$.
\begin{lem}\label{L3.8}
Let $B\in\mathbf{B}$. Assume that $\nu\in\scr{P}_0$ and $\nu=h\mu$ with $\|h\phi_0^{-1}\|_\infty<\infty$. Then there exists a constant $c>0$ such that
\begin{equation}\begin{split}\label{6L3.8+}
\sup_{T\geq t}\E^\nu[\W_2(\mu_{t,r}^B,\mu_t^B)^2|T<\si_\tau^B]\leq c r,\quad t\geq1,\,r\in(0,1],
\end{split}\end{equation}
and
\begin{equation}\begin{split}\label{L3.8+1}
\sup_{T\ge t}\E^{\nu}[\W_2(\mu_{t,r,r}^B,\mu_t^B)^2|T<\si_\tau^B]\le cr,\quad t\geq1,\,r\in(0,1].
\end{split}\end{equation}
\end{lem}
\begin{proof}
By Lemma \ref{L3.3}, it suffices to prove the case when $T=t$ instead of $T\ge t$. By the triangle inequality of $\W_2$, we have
\begin{equation}\label{1L3.8}
\W_2(\mu_{t,r,r}^B,\mu_t^B)^2\le 2\W_2(\mu_{t,r,r}^B,\mu_{t,r}^B)^2+2\W_2(\mu_{t,r}^B,\mu_t^B)^2,\quad r\in(0,1],\,t>0.
\end{equation}

To estimate $\W_2(\mu_{t,r,r}^B,\mu_{t,r}^B)$, since
\begin{equation*}\begin{split}\label{2L3.8}
\W_2(\mu,\nu)^2&:=\inf_{\pi\in\mathscr{C}(\mu,\nu)}\int_{M\times M}\rho(x,y)^2\pi(\d x,\d y)\\
& \le {\rm diam}(M)^2\inf_{\pi\in\mathscr{C}(\mu,\nu)}\pi(\{(x,y):x\neq y\})\\
&=\ff 1 2 {\rm diam}(M)^2\|\mu-\nu\|_{\rm var}^2,\quad \mu,\nu\in\scr{P},
\end{split}\end{equation*}
where $\|\cdot\|_{\rm var}$ is the total variation norm, we have
\begin{equation}\begin{split}\label{3L3.8}
&\W_2(\mu_{t,r,r}^B,\mu_{t,r}^B)^2\le \ff 1 2{\rm diam}(M)^2\|\mu_{t,r,r}^B-\mu_{t,r}^B\|_{\rm var}\\
&=\ff 1 2{\rm diam}(M)^2\mu_0(|\rho_{t,r,r}^B-\rho_{t,r}^B|)\le {\rm diam}(M)^2 r,\quad r\in(0,1],\,t>0.
\end{split}\end{equation}

Next, we consider $\W_2(\mu_{t,r}^B,\mu_t^B)$. For every $t>0$ and every $r\in(0,1]$, let
$$\pi(\d x,\d y):=\mu_t^B(\d x)p_r^0(x,y)\,\mu_0(\d y).$$
It is immediate to check that $\pi\in\mathscr{C}(\mu_t^B,\mu_{t,r}^B)$. Then
\begin{equation*}\label{4L3.8}
\W_2(\mu_t^B,\mu_{t,r}^B)^2\le\int_M\E^x[\rho(x,X_r^0)^2]\,\mu_t^B(\d x),\quad r\in(0,1],\,t>0.
\end{equation*}
According to the proof of \cite[Lemma 3.8]{eW2}, we have
\begin{equation*}\begin{split}\label{5L3.8}
&\E^\nu[1_{\{t<\si_\tau^B\}}\W_2(\mu_{t,r}^B,\mu_t^B)^2]\\
&\le\ff{c_1 r}t\int_0^t \E^\nu[1_{\{t<\si_\tau^B\}}\log\{1+\phi_0^{-1}(X_s^B)\}]\,\d s,\quad r\in(0,1],\,t>0
\end{split}\end{equation*}
for some constant $c_1>0$. By \eqref{2L3.2} with $f$ replaced by $\log(1+\phi_0^{-1})$, H\"{o}lder's inequality, \eqref{LT} and the contractivity of $P_t^0$ in $L^q(\mu_0)$ for all $q\geq1$, we get for every $p\in(\ff 3 2,3)$,
\begin{equation}\begin{split}\label{6L3.8}
&\E^\nu[1_{\{t<\si_\tau^B\}}\log\{1+\phi_0^{-1}(X_s^B)\}]\\
&=\int_0^\infty\!\!\!\int_0^\infty\e^{-\ll_0 (u+v)}\mu_0\big(h\phi_0^{-1} P_v^0(\log\{1+\phi_0^{-1}\}P_u^0\phi_0^{-1})\big)\,\P_{S_{t-s}^B}(\d u)\P_{S_s^B} (\d v)\\
&\le\|h\phi_0^{-1}\|_\infty\|\phi_0^{-1}\|_{L^p(\mu_0)}\|\log(1+\phi_0^{-1})\|_{L^{\ff p{p-1}}(\mu_0)}\int_0^\infty\!\!\!\int_0^\infty \e^{-\ll_0 (u+v)}\,\P_{S_{t-s}^B}(\d u)\P_{S_s^B}(\d v)\\
&=\e^{-B(\ll_0) t}\|h\phi_0^{-1}\|_\infty\|\phi_0^{-1}\|_{L^p(\mu_0)}\|\log(1+\phi_0^{-1})\|_{L^{\ff p{p-1}}(\mu_0)}\\
&\le\e^{-B(\ll_0) t}\|h\phi_0^{-1}\|_\infty\|\phi_0^{-1}\|_{L^p(\mu_0)}\|\phi_0^{-1}\|_{L^{\ff p {p-1}}(\mu_0)},\quad 0<s<t,
\end{split}\end{equation}
where the elementary inequality $\log (1+u)\leq u$ for every $u\geq0$ was applied in the last inequality.

Thus, by \eqref{SDHT}, \eqref{PHI} and \eqref{6L3.8}, we find constants $c_2,c_3>0$ such that
\begin{equation*}\begin{split}
&\E^\nu[\W_2(\mu_{t,r}^B,\mu_t^B)^2|t<\si_\tau^B]
=\ff{\E^\nu[1_{\{t<\si_\tau^B\}}\W_2(\mu_{t,r}^B,\mu_t^B)^2]}{\Pp^\nu(t<\si_\tau^B)}\\
&\leq c_2\frac{r\e^{B(\lambda_0)t}}{t\mu(\phi_0)\nu(\phi_0)}\int_0^t \E^\nu\big[1_{\{t<\si_\tau^B\}}\log\{1+\phi_0^{-1}(X_s^B)\}\big]\d s\\
&\leq c_3 r,\quad t\geq1,\,r\in(0,1],
\end{split}\end{equation*}
which finish the proof of \eqref{6L3.8+}. Combining this with \eqref{3L3.8} and \eqref{1L3.8}, we immediately prove \eqref{L3.8+1}.
\end{proof}

With these preparations, we are ready to prove Theorem \ref{T1.2} now. We divide the proof into two parts.
\begin{proof}[Proof of Theorem \ref{T1.2}(1)]
Assume that $\alpha\in(0,1]$, $B\in\BB^\alpha$ and $d<2(1+\alpha)$.

\textbf{Step 1}. By \eqref{EIG}, \eqref{equ-B} and \eqref{BR}, we have some constant $c_0>0$ such that
$$\sum_{m=1}^\infty\ff 1 {(\ll_m-\ll_0)[B(\ll_m)-B(\ll_0)]}\leq c_0\sum_{m=1}^\infty\frac{1}{ m^{2(1+\alpha)/d}}<\infty,$$
which proves the second inequality in \eqref{0T1.2}.

Next we prove the first inequality in \eqref{0T1.2} and the second assertion of Theorem \ref{T1.2}(1).

\textbf{Step 2}. Assume that $\nu=h\mu$ with $h\le C\phi_0$ for some constant $C>0$.

(i) Let $d<6\aa-2$  for some $\aa\in(\ff 1 2,1]$. For every $\zeta\in(1,\ff 4 d \wedge\ff{2(d+2)}{d^2+2(1-\aa)d})$, let $r_t:=t^{-\zeta}$, $t>0$.
Applying the triangle inequality of $\W_2$, we see that, for any $\vv>0$,
\begin{equation}\begin{split}\label{1P3.1}
\E^\nu[\W_2(\mu_t^B,\mu_0)^2|t<\si_\tau^B]&\le(1+\vv)\E^\nu[\W_2(\mu_{t,r_t,r_t}^B,\mu_0)^2|t<\si_\tau^B]\\
&\quad+(1+\vv^{-1})\E^\nu[\W_2(\mu_{t,r_t,r_t}^B,\mu_t^B)^2|t<\si_\tau^B],\quad t\geq1.
\end{split}\end{equation}

On the one hand, \eqref{L3.8+1} in Lemma \ref{L3.8} implies that
\begin{equation}\label{1P3.2-}
\E^\nu[\W_2(\mu_{t,r_t,r_t}^B,\mu_t^B)^2|t<\si_\tau^B]\le c_1 t^{-\zeta},\quad t\geq1,
\end{equation}
for some constant $c_1>0$.

On the other hand, using the fact that $\rho_{t,r_t,r_t}^B-1=(1-r_t)(\rho_{t,r_t}^B-1)$ for every $t\geq1$ and applying the inequality \eqref{W2UB}, we have
\begin{equation}\begin{split}\label{2P3.1}
t\E^\nu[\W_2(\mu_{t,r_t,r_t}^B,\mu_0)^2|t<\si_\tau^B]&\le t\E^\nu\left[\int_M\ff{|\nn (-\mathcal{L}_0)^{-1}(\rho_{t,r_t,r_t}^B-1)|^2}{\scr{M}(\rho_{t,r_t,r_t}^B,1)}\,\d \mu_0 \Big|t<\si_\tau^B\right]\\
&\leq(1-r_t)^2\big[ {\rm I}_1(t)+{\rm I}_2(t)\big],\quad t\geq1,
\end{split}\end{equation}
where we set, for every $t\geq1$,
\begin{align*}
&{\rm I}_1(t):=t\E^\nu\big[\mu_0\big(|\nn (-\mathcal{L}_0)^{-1}(\rho_{t,r_t}^B-1)|^2\big)|t<\si_\tau^B\big],\\
&{\rm I}_2(t):=t\E^\nu\big[\mu_0\big(|\nn (-\mathcal{L}_0)^{-1}(\rho_{t,r_t}^B-1)|^2|\mathscr{M}(\rho_{t,r_t,r_t}^B,1)^{-1}-1|\big)|t<\si_\tau^B\big].
\end{align*}
So, it remains to estimate both ${\rm I}_1$ and ${\rm I}_2$.

For ${\rm I}_1$, since $B\in \BB^\aa$ for some $\aa\in(\ff 1 2,1]$ and $d<6\aa-2$, Lemma \ref{L3.2}(1) yields that
\begin{equation*}\begin{split}
\left|{\rm I}_1(t)-\sum_{m=1}^\infty\ff {2\e^{-2(\ll_m-\ll_0)t^{-\zeta}}}{[B(\ll_m)-B(\ll_0)](\ll_m-\ll_0)}\right|&\le c_2 \Big\{ t^{-1+\frac{(d-2)^+}{2}\zeta}+\zeta1_{\{d=2\}}t^{-1}\log t \Big\},\quad t\geq1,
\end{split}\end{equation*}
for some constant $c_2>0$. Due to that $\frac{(d-2)^+}{2}\zeta<1$ for every $\zeta\in(1,\ff 4 d \wedge\ff{2(d+2)}{d^2+2(1-\aa)d})$, we have
\begin{equation}\begin{split}\label{2P3.1-I1}
\limsup_{t\to\infty}{\rm I}_1(t) \le\sum_{m=1}^\infty\ff 2{[B(\ll_m)-B(\ll_0)](\ll_m-\ll_0)}.
\end{split}\end{equation}
For ${\rm I}_2$, by H\"{o}lder's inequality, Lemma \ref{L3.6} with $p\in(1,\frac{2d+6}{3d+4-2\alpha}\wedge \ff {d+2}{d+1})$, Lemma \ref{L3.7} with $q=\ff p {p-1}$, and Lemma \ref{nu0}, we derive that
\begin{equation}\begin{split}\label{2P3.1-I2}
\limsup_{t\to\infty}{\rm I}_2(t)&\leq\limsup_{t\to\infty}\Big\{t\big(\E^\nu\big[\mu_0\big(|\nn(- \mathcal{L}_0)^{-1}(\rho_{t,r_t}^B-1)|^{2p}|t<\si_\tau^B\big)\big]\big)^{\ff 1 p}\\
&\quad \times \big(\E^\nu\big[\mu_0\big(|\mathscr{M}(\rho_{t,r_t,r_t}^B,1)^{-1}-1|^q\big)|t<\si_\tau^B\big]\big)^{\ff 1 q}\Big\}\\
&=0.
\end{split}\end{equation}

Gathering \eqref{2P3.1-I2}, \eqref{2P3.1-I1}  and \eqref{2P3.1} together, we derive that
\begin{equation}\begin{split}\label{2P3.1+}
\limsup_{t\to\infty}\big\{t\E^\nu[\W_2(\mu_{t,r_t,r_t}^B,\mu_0)^2|t<\si_\tau^B]\big\}\le
\sum_{m=1}^\infty\ff 2{[B(\ll_m)-B(\ll_0)](\ll_m-\ll_0)}.
\end{split}\end{equation}

Thus, by \eqref{1P3.1}, \eqref{1P3.2-} and \eqref{2P3.1+}, we arrive at
$$\limsup_{t\to\infty}\big\{t\sup_{T\ge t}\E^\nu[\W_2(\mu_t^B,\mu_0)^2|T<\si_\tau^B]\big\}\le\sum_{m=1}^\infty\ff 2 {(\ll_m-\ll_0)[B(\ll_m)-B(\ll_0)]}.$$

(ii) Let
$d\geq6\alpha-2$ for some $\aa\in(0,1]$. For every $\zeta\in(1,\ff 6{d+2-6\aa})$, let $r_t:=t^{-\zeta}$, $t>0$. By \eqref{1P3.1}, we have
\begin{equation}\begin{split}\label{I21}
\E^\nu[\W_2(\mu_t^B,\mu_0)^2|t<\si_\tau^B]&\le2\E^\nu[\W_2(\mu_{t,r_t}^B,\mu_0)^2|t<\si_\tau^B]\\
&\quad+2\E^\nu[\W_2(\mu_{t,r_t}^B,\mu_t^B)^2|t<\si_\tau^B],\quad t\ge1.
\end{split}\end{equation}
According to \eqref{6L3.8+}, there exists a constant $c_3>0$ such that
\begin{equation}\label{I22}
t\E^\nu[\W_2(\mu_{t,r_t}^B,\mu_t^B)^2|t<\si_\tau^B]\le c_3 t^{1-\zeta},\quad t\ge 1.
\end{equation}
Since $d\geq 6\alpha-2$,
by \eqref{EIG}, \eqref{W2UB-ledoux} and \eqref{0L3.2}, we have for every $\gg>\ff{d+2-6\aa}6,$
\begin{equation}\begin{split}\label{I23}
&t\E^\nu[\W_2(\mu_{t,r_t}^B,\mu_0)^2|t<\si_\tau^B]\le 4t\E^\nu[\mu_0(|\nabla(-\mathcal{L}_0)^{-1}(\rho_{t,r_t}^B-1)|^2)|t<\si_\tau^B]\\
&\le\sum_{m=1}^\infty\ff {c_4}{(\ll_m-\ll_0)[B(\ll_m)-B(\ll_0)]}+c_4t^{\gg\zeta-1},\quad t\geq1,
\end{split}\end{equation}
for some constant $c_4>0$. By
\eqref{I21}, \eqref{I22}, \eqref{I23} and Lemma \ref{L3.3}, letting $t\to\infty$, we find a constant $c_5>0$ such that
\begin{equation}\begin{split}\label{I24}
\limsup_{t\to\infty}\Big\{t\sup_{T\ge t}\E^\nu[\W_2(\mu_t^B,\mu_0)^2|T<\si_\tau^B]\Big\}\le \sum_{m=1}^\infty\ff {c_5}{(\ll_m-\ll_0)[B(\ll_m)-B(\ll_0)]}.
\end{split}\end{equation}

\textbf{Step 3}. In general, let $t>\varepsilon>0$. Consider
$$\mu_t^{B,\vv}:=\ff 1 {t-\vv}\int_\vv^t \delta_{X_s^B}\d s.$$
Then, we immediately derive the following inequality, i.e.,
\begin{equation}\begin{split}\label{3P3.1}
&\W_2(\mu_t^{B,\vv},\mu_t^B)^2\le{\rm diam}(M)^2\|\mu_t^B-\mu_t^{B,\vv}\|_{\rm var}\\
&\le {\rm diam}(M)^2\Big[\int_{\vv}^t\left(\ff 1{t-\vv}-\ff 1 t\right)\,\d s+\ff{1}{t}\int_0^\vv\d s\Big]\\
&=2{\rm diam}(M)^2\vv t^{-1}.
\end{split}\end{equation}

Let $\nu_\vv^B=h_\vv^B\mu$ with
$$h_\vv^B(y)=\ff 1{\Pp^\nu(\vv<\si_\tau^B)}\int_M {p}_\vv^{D,B}(x,y)\,\nu(dx),
\quad y\in M.$$
From \eqref{EIG}, \eqref{SDHK}, \eqref{EIG0UB}, \eqref{SDHT-L},  \eqref{BR} and \eqref{equ-B}, it is easy to see that
\begin{equation*}\begin{split}
h_\vv^B(y)&\leq \frac{\|\phi_0\|_\infty}{\nu(\phi_0)}\phi_0(y)\int_M\Big(\phi_0(x)+\sum_{m=1}^\infty\e^{-[B(\lambda_m)-B(\lambda_0)]\varepsilon}
\phi_m(x)(\phi_m\phi_0^{-1})(y)\Big)\,\nu(\d x)\\
&\leq c_6\phi_0(y)\Big(1+\sum_{m=1}^\infty\frac{m^{(d+1)/d}}{\e^{c_7 \varepsilon m^{2\alpha/d}}}\Big)<\infty,\quad \varepsilon>0,\,y\in M,
\end{split}\end{equation*}
for some constants $c_6,c_7>0$. By the Markov property,
\begin{equation*}\begin{split}
&\E^\nu[1_{\{t<\si_\tau^B\}}\W_2(\mu_t^{B,\vv},\mu_0)^2]
=\E^\nu\big\{1_{\{\vv<\si_\tau^B\}}\E^{X_\vv^B}[1_{\{t-\vv<\si_\tau^B\}}\W_2(\mu_{t-\vv}^{B},\mu_0)^2]\big\}\\
&=\Pp^\nu(\vv<\si_\tau^B)\E^{\nu_\vv^B}[1_{\{t-\vv<\si_\tau^B\}}\W_2(\mu_{t-\vv}^{B},\mu_0)^2]\\
&=\Pp^\nu(\vv<\si_\tau^B)\Pp^{\nu_\vv^B}(t-\vv<\si_\tau^B)\E^{\nu_\vv^B}[\W_2(\mu_{t-\vv}^{B},\mu_0)^2|t-\vv<\si_\tau^B].
\end{split}\end{equation*}
By \eqref{SDHK} and \eqref{SDHT}, it is easy to verify
$$\lim_{t\to\infty}\ff{\Pp^{\nu_\vv^B}(t-\vv<\si_\tau^B)\Pp^\nu(\vv<\si_\tau^B)}{\Pp^\nu(t<\si_\tau^B)}=1.$$
Thus, if $d<6\aa-2$ for some $\aa\in(\ff 1 2,1]$, then by part (i) in \textbf{Step 2}, we deduce that
\begin{equation}\begin{split}\label{3P3.1+}
&\limsup_{t\to\infty}\left\{t\E^\nu[\W_2(\mu_t^{B,\vv},\mu_0)^2|t<\si_\tau^B]\right\}\\
&=\limsup_{t\to\infty}\ff{\Pp^{\nu_\vv^B}(t-\vv<\si_\tau^B)\Pp^\nu(\vv<\si_\tau^B)}{\Pp^\nu(t<\si_\tau^B)}
\left\{t\E^{\nu_\vv^B}[\W_2(\mu_{t-\vv}^B,\mu_0)^2|t-\vv<\si_\tau^B]\right\}\\
&\le\sum_{m=1}^\infty\ff 2 {[B(\ll_m)-B(\ll_0)](\ll_m-\ll_0)}.
\end{split}\end{equation}

Therefore, according \eqref{3P3.1+}, \eqref{3P3.1} and the triangle inequality of $\W_2$, we have
\begin{equation*}\begin{split}
&\limsup_{t\to\infty}\left\{t\E^\nu[\W_2(\mu_t^B,\mu_0)^2|t<\si_\tau^B]\right\}\\
&\leq \limsup_{t\to\infty}\Big\{(1+\vv^{\ff 1 2})t\E^\nu[\W_2(\mu_t^{B,\varepsilon},\mu_0)^2|t<\si_\tau^B] \\
&\quad+ (1+\vv^{-\ff 1 2})t\E^\nu[\W_2(\mu_t^B,\mu_t^{B,\varepsilon})^2|t<\si_\tau^B]\Big\}\\
&\le(1+\vv^{\ff 1 2})\sum_{m=1}^\infty\ff 2 {[B(\ll_m)-B(\ll_0)](\ll_m-\ll_0)}+c_8\vv(1+\vv^{-\ff 1 2}),\quad \vv\in(0,1),
\end{split}\end{equation*}
for some constant $c_8>0$.  Letting $\vv\to 0$, we derive the the inequality in \eqref{0T1.2} with $c=1$.

Finally, by a similar argument in \textbf{Step 3} above, we can prove that \eqref{I24} also holds for every initial distribution $\nu\in\mathscr{P}_0$ when $d\ge 6\aa-2$ for some $\alpha\in(0,1]$.
\end{proof}

\begin{proof}[Proof of (2) and (3) in Theorem \ref{T1.2}]
Assume that $B\in\BB^\aa$ for some $\aa\in(0,1]$ and $\nu\in\mathscr{P}_0$. Let $d\ge 2(1+\aa)$. By \eqref{3P3.1}, it suffices to prove the desired estimates for $\mu_t^{B,1}$ replacing $\mu_t^B$. Due to this, we may assume that $\nu=h\mu$ with $\|h\phi_0^{-1}\|_\infty<\infty$.

By \eqref{0L3.3}
and \eqref{W2UB-ledoux}, for any $\gamma>\ff{d+2-6\aa} 6$, we find a constant $c>0$ such that
\begin{equation*}\begin{split}
&t\E^\nu[\W_2(\mu_{t,r}^B,\mu_0)^2|T<\si_\tau^B]\\
&\le c\Big(r^{-\ff{d-2(1+\aa)}{2}}+1_{\{d=2(1+\aa)\}}\log r^{-1}+t^{-1}r^{-\gamma}\Big),\quad T\ge t\ge 1,\,r\in(0,1).
\end{split}\end{equation*}
Then, by \eqref{6L3.8+} and the triangle inequality of $\W_2$, we can find a function $a:(\ff{d+2-6\aa}6,\infty)\to(0,\infty)$ such that, for every $\gamma>\ff{d+2-6\aa}6$,
\begin{equation}\begin{split}\label{4P3.1}
\E^\nu[\W_2(\mu_t^B,\mu_0)^2|T<\si_\tau^B]&\leq 2\E^\nu[\W_2(\mu_{t,r}^B,\mu_0)^2|T<\si_\tau^B] +2 \E^\nu[\W_2(\mu_t^B,\mu_{t,r}^B)^2|T<\si_\tau^B]\\
&\le a(\gamma)\Big(t^{-1}r^{-\ff{d-2(1+\aa)}{2}}+t^{-1}1_{\{d=2(1+\aa)\}}\log r^{-1}\\
&\quad +t^{-2}r^{-\gamma}+r\Big),\quad T\ge t\ge 1,\,r\in(0,1).
\end{split}\end{equation}

(i) Let $d>2(1+\aa)$. Taking $\gamma=\ff{7d-18\aa-4} {12}$ and $r=t^{-\ff 2{d-2\aa}}$ for any $t>1$ in \eqref{4P3.1}, we have
$$\sup_{T\ge t}\E^\nu[\W_2(\mu_t^B,\mu_0)^2|T<\si_\tau^B]\le c_2\Big(t^{-\frac{2}{d-2\alpha}}+t^{-\frac{5d-6\alpha+4}{6(d-2\alpha)}}\Big)\leq c_3 t^{-\frac{2}{d-2\alpha}},\quad t\ge 2,$$
for some constants $c_2,c_3>0$, which leads to the desired result in Theorem \ref{T1.2}(2).

(ii) Let $d=2(1+\aa)$. Taking $\gamma=\ff 1 2$ and $r=t^{-1}$ for any $t>1$ in \eqref{4P3.1}, we obtain
$$\sup_{T\ge t}\E^\nu[\W_2(\mu_t^B,\mu_0)^2|T<\si_\tau^B]\le c_1t^{-1}\log t,\quad t\ge 2,$$
for some constant $c_1>0$, which proves Theorem \ref{T1.2}(3).
\end{proof}

\section{Lower bounds: proof of Theorem \ref{T1.1}}

In this section, we follow closely \cite[Section 4]{eW2}. The overall idea is contained in Lemma \ref{W2-LB} below. In each of the following lemmas,  we assume additionally the condition that $\nu\in\mathscr{P}_0$ and $B\in\mathbf{B}$
 without mentioning it every time.

Recall that $\mu_{t,r}^B=\rho_{t,r}^B\mu_0$, $t,r>0$. By the same approach, taking $f=(-\mathcal{L}_0)^{-1}(\rho_{t,r}^B-1)$ instead of the original $f$ in the proof of \cite[Lemma 4.2]{eW2}, we can established the next lemma. The detailed proof is omitted here.
\begin{lem}\label{W2-LB}
There exists some constant $c>0$ such that, for every $t,r>0$,
$$\W_2(\mu_{t,r}^B,\mu_0)^2\geq  \mu_0\big(|(-\mathcal{L}_0)^{-1}(\rho_{t,r}^B-1)|^2\big)-c\|\rho_{t,r}^B-1\|_\infty^{7/3}\big(1+\|\rho_{t,r}^B-1\|^{1/3}_\infty\big).$$
\end{lem}

From Lemma \ref{W2-LB}, it is clear that, in order to obtain a sharp lower bound on $\W_2(\mu_{t,r}^B,\mu_0)^2$, we should deal with $\mu_0\big(|(-\mathcal{L}_0)^{-1}(\rho_{t,r}^B-1)|^2\big)$ and $\|\rho_{t,r}^B-1\|_\infty$ carefully next.

\begin{lem}\label{L4.3}
Let $\nu=h\mu$ with $\|h\phi_0^{-1}\|_\infty<\infty$. Then, for every $r>0$, there exists a constant $c_r>0$ such that
\begin{equation*}\begin{split}
&\sup_{T\ge t}\left|t\E^{\nu_0}\big[\mu_0\big(|\nn (-\mathcal{L}_0)^{-1}(\rho_{t,r}^B-1)|^2\big)|T<\si_\tau^B\big] -2\sum_{m=1}^\infty\ff{\e^{-2(\ll_m-\ll_0)r}}{(\ll_m-\ll_0)[B(\ll_m)-B(\ll_0)]}\right|\\
&\le c_r t^{-1},\quad t\ge 1.
\end{split}\end{equation*}
\end{lem}
\begin{proof}
Let ${\rm J_i}$, $i=1,2,3$, be defined in \eqref{5L3.2}. By \eqref{PI0}, \eqref{EIG0UB} and $\|\phi_m\phi_0^{-1}\|_{L^2(\mu_0)}=1$ for all $m\in\mathbb{N}$, we find a constant $c_1>0$ such that, for every $u,v,w>0$ and every $m\in\mathbb{N}$,
\begin{align*}
|{\rm J_1}(u,v,w)|&\le
\|h\phi_0^{-1}\|_\infty\|P_u^0-\mu_0\|_{L^\infty(\mu_0)\rightarrow L^\infty(\mu_0)}\|\phi_m\phi_0^{-1}\|_\infty^2\\
&\quad\times\|P_w^0-\mu_0\|_{L^1(\mu_0)\rightarrow L^1(\mu_0)}\|\phi_0^{-1}\|_{L^1(\mu_0)}\\
&\le c_1\|\phi_m\phi_0^{-1}\|_\infty^2 \e^{-(\ll_1-\ll_0)(u+w)},\\
|{\rm J_2}(u,v,w)|&\le\|\phi_0\|_\infty\e^{-(\ll_m-\ll_0)v}\|h\phi_0^{-1}\|_\infty\|P_u^0-\mu_0\|_{L^\infty(\mu_0)\rightarrow L^\infty(\mu_0)}\\
&\le c_1\e^{-(\ll_1-\ll_0)(u+v)},\\
|{\rm J_3}(u,v,w)|&\le\|\phi_0\|_\infty\e^{-(\ll_m-\ll_0)v}\|\phi_m\phi_0^{-1}\|_\infty^2\|P_w^0-\mu_0\|_{L^1(\mu_0)\rightarrow L^1(\mu_0)}\|\phi_0^{-1}\|_{L^1(\mu_0)}\\
&\le c_1\|\phi_m\phi_0^{-1}\|_\infty^2\e^{-(\ll_1-\ll_0)(v+w)}.
\end{align*}
Let $r>0$ be fixed. Substituting these estimates into \eqref{8L3.2} (whose proof does not need further assumptions on $B$) and applying \eqref{EIG0} and \eqref{EIG0UB}, we find a constant $c_r>0$ such that the desired result holds.
\end{proof}

Since $\rho_{t,r}^B=\frac{1}{t}\int_0^tp_r^0(X_s^B,\cdot)$\,\d s by \eqref{MTR} and $\|\rho_{t,r}^B\|_{L^\infty(\mu_0)}\leq\|p_t^0\|_{L^\infty(\mu_0\times\mu_0)}<\infty$ by \eqref{PI0}, $t,r>0$, by the same proof of \cite[Lemma 4.5]{eW2}, we obtain the following lemma. The details are left to the reader.
\begin{lem}\label{LpBD}
Assume that $\nu=h\mu$ and $\|h\phi_0^{-1}\|_\infty<\infty$. Then, for every $q\geq2$ and every $r>0$, there exists some constant $c_{q,r}>0$ depending on $q,r$ such that
$$\|\nabla(-\mathcal{L}_0)^{-1}(\rho_{t,r}^B-1)\|_{L^{2q}(\mu_0)}\leq c_{q,r},\quad t>0.$$
\end{lem}

As for the estimate on $\|\rho_{t,r}^B-1\|_\infty$, we derive the following result.
\begin{lem}\label{L4.2}
For any $r>0$ and any $\nu=h\mu$ with $\|h\phi_0^{-1}\|_\infty<\infty$, there exists a constant $c_r>0$ depending on $r$ such that
$$\sup_{T\ge t}\E^\nu[\|\rho_{t,r}^B-1\|_\infty^4|T<\si_\tau^B]\le c_r t^{-2},\quad t\ge 1.$$
\end{lem}

\begin{proof} Let $r>0$ be fixed. Recall that $\nu_0=\frac{\phi_0}{\mu(\phi_0)}\mu$.
By Lemma \ref{L3.3} and Lemma \ref{nu0}, 
it is enough to prove the assertion for $\nu=\nu_0$ and $T=t$ replacing $T\ge t$, i.e.,
\begin{equation}\label{1L4.2}
\E^{\nu_0}[\|\rho_{t,r}^B-1\|_\infty^4|t<\si_\tau^B]\le c_rt^{-2},\quad t\ge 1,
\end{equation}
for some constant $c_r>0$.

For each $m\in\mathbb{N}$ and every $0<r_1<r_2<s$, let
$$\eta_s(r_1,r_2):=\E^{\nu_0}\big[1_{\{s<\sigma_\tau^B\}}(\phi_m\phi_0^{-1})(X_{r_1}^B)(\phi_m\phi_0^{-1})(X_{r_2}^B)|\sigma(X_r^B:\ 0\leq r\leq r_1)\big].$$
Similar as the proof of \eqref{I} (which still works for $B\in\mathbf{B}$), we find some constant $c_1>0$ such that
\begin{equation*}\begin{split}
&\E^{\nu_0}[1_{\{r_1<\si_\tau^B\}}|\eta_s(r_1,r_2)|^2]\leq c_1\|\phi_m\phi_0^{-1}\|_\infty^4 \e^{-B(\lambda_0)(2s-r_1)}\e^{-2[B(\lambda_m)-B(\lambda_0)](r_2-r_1)},
\end{split}\end{equation*}
By \eqref{8L3.5} (whose proof only depends on the Markov property and Fubini's theorem) and \eqref{SDHT-L},
$$\E^{\nu_0}[|\psi_m^B(t)|^4|t<\si_\tau^B]=\ff{\E^{\nu_0}[|\psi_m^B(t)|^4 1_{\{t<\si_\tau^B\}}]}{\Pp^{\nu_0}(t<\si_\tau^B)}\le c_2\|\phi_m\phi_0^{-1}\|_\infty^4 t^{-2},\quad m\in\mathbb{N},\,t\geq1,$$
for some constant $c_2>0$. Combining this with \eqref{MTR}, H\"{o}lder's inequality, Fubini's theorem and \eqref{equ-B},  we have
\begin{equation*}\begin{split}
&\E^{\nu_0}[\|\rho_{t,r}^B-1\|_\infty^4|t<\si_\tau^B]=\E^{\nu_0}\left[\Big\|\sum_{m=1}^\infty \e^{-(\lambda_m-\lambda_0)r}\psi_m^B(t)\phi_m\phi_0^{-1}\Big\|_\infty^4\Big| t<\sigma^B_\tau\right]\\
&\le\left(\sum_{m=1}^\infty \e^{-(\ll_m-\ll_0)r}\|\phi_m\phi_0^{-1}\|_\infty^{\ff 4 3}\right)^3\sum_{m=1}^\infty \e^{-(\ll_m-\ll_0)r} \E^{\nu_0}[|\psi_m^B(t)|^4| t<\si_\tau^B]\\
&\le c_2t^{-2}\left(\sum_{m=1}^\infty \e^{-(\ll_m-\ll_0)r}\|\phi_m\phi_0^{-1}\|_\infty^{\ff 4 3}\right)^3 \sum_{m=1}^\infty\e^{-(\ll_m-\ll_0)r}\|\phi_m\phi_0^{-1}\|_\infty^4,\quad t\geq1.
\end{split}\end{equation*}
 By making use of \eqref{EIG} and \eqref{EIG0UB}, we immediately prove the desired inequality \eqref{1L4.2} for some constant $c_r>0$.
\end{proof}

Now we are ready to prove the second main result.
\begin{proof}[Proof of Theorem \ref{T1.1}]
(1) Let $B\in\mathbf{B}$. According to \eqref{3P3.1}, it suffices to prove \eqref{1T1.1} for the case when $\nu=h\mu$ with $\|h\phi_0^{-1}\|_\infty<\infty$. By Lemma \ref{W2-LB}, we have some constant $c_1>0$ such that
\begin{equation}\begin{split}\label{1P4.1}
&t\E^\nu[\W_2(\mu_{t,r}^B,\mu_0)^2|T<\si_\tau^B]\ge t\E^{\nu}[1_{\{\|\rho_{t,r}^B-1\|_\infty\le\vv\}}\W_2(\mu_{t,r}^B,\mu_0)^2|T<\si_\tau^B]\\
&\ge t\E^\nu\big[1_{\{\|\rho_{t,r}^B-1\|_\infty\le\vv\}}\mu_0\big(|\nn (-\mathcal{L}_0)^{-1}(\rho_{t,r}^B-1)|^2\big)|T<\si_\tau^B\big]-c_1\vv^2\\
&= t\E^\nu\big[\mu_0\big(|\nn (-\mathcal{L}_0)^{-1}(\rho_{t,r}^B-1)|^2\big)|T<\si_\tau^B\big]-c_1\vv^2\\
&\quad-t\E^\nu\big[1_{\{\|\rho_{t,r}^B-1\|_\infty>\vv\}}\mu_0\big(|\nn (-\mathcal{L}_0)^{-1}(\rho_{t,r}^B-1)|^2\big)|T<\si_\tau^B\big],\quad \vv\in(0,1),\,r>0,T\ge t\geq1.
\end{split}\end{equation}
By  H\"{o}lder's inequality, Lemma \ref{LpBD} with $p=3$,  Chebyshev's inequality,  and Lemma \ref{L4.2}, we find some constants $c_2,c_3>0$ such that
\begin{equation}\begin{split}\label{1P4.1+}
&t\E^\nu\big[1_{\{\|\rho_{t,r}^B-1\|_\infty>\vv\}}\mu_0\big(|\nn (-\mathcal{L}_0)^{-1}(\rho_{t,r}^B-1)|^2\big)|T<\si_\tau^B\big]\\
&\le c_2 t\{\Pp^\nu(\|\rho_{t,r}^B-1\|_\infty>\vv|T<\si_\tau^B)\}^{\ff 2 3}\\
&\le c_2 t\vv^{-\ff 8 3}\{\E^\nu(\|\rho_{t,r}^B-1\|_\infty^4|T<\si_\tau^B)\}^{\ff 2 3}\\
&\le c_3\vv^{-\ff 8 3}t^{-\ff 1 3},\quad r,\varepsilon>0,\, T\ge t\geq1.
\end{split}\end{equation}
Plugging  \eqref{1P4.1+} into \eqref{1P4.1} and applying Lemma \ref{L4.3}, for every $r>0$, we have some constant $c_4>0$ such that
\begin{equation*}\begin{split}
&t\E^\nu[\W_2(\mu_{t,r}^B,\mu_0)^2|T<\si_\tau^B]\\
&\ge t\E^\nu\big[\mu_0\big(|\nn (-\mathcal{L}_0)^{-1}(\rho_{t,r}^B-1\big)|^2)|T<\si_\tau^B\big]-\vv_t\\
&\ge 2\sum_{m=1}^\infty \ff{\e^{-2(\ll_m-\ll_0)r}}{(\ll_m-\ll_0)[B(\ll_m)-B(\ll_0)]}-\vv_t-c_4 t^{-1},\quad T\ge t\ge 1,
\end{split}\end{equation*}
where we set
$$\vv_t:=\inf_{\vv\in(0,1)}\big(c_1\vv^2+c_3\vv^{-\ff 8 3}t^{-\ff 1 3}\big).$$
It is easy to see that $\vv_t\to 0 $ as $t\to\infty.$ Hence,
$$\liminf_{t\to\infty}\inf_{T\ge t}\{t\E^\nu[\W_2(\mu_{t,r}^B,\mu_0)^2|T<\si_\tau^B]\}\ge 2\sum_{m=1}^\infty\ff{\e^{-2(\ll_m-\ll_0)r}}{(\ll_m-\ll_0)[B(\ll_m)-B(\ll_0)]},\quad r>0.$$
Combining this with \cite[Lemma 4.6]{eW2}, i.e., there exist constants $\kappa_1,\kappa_2>0$ such that, for any Borel probability measures $\nu_1,\nu_2$ on $\mathring{M}$,
\begin{equation}\label{W-contra}
\W_2(\nu_1 P_u^0,\nu_2P_u^0)\leq \kappa_1\e^{\kappa_2 u}\W_2(\nu_1,\nu_2),\quad u\geq0,
\end{equation}
and recalling that $\mu_{t,r}^B=\mu_t^B P_r^0$, we find a constant $c_5>0$ such that
\begin{equation}\begin{split}\label{W2-r}
&\liminf_{t\to\infty}\inf_{T\ge t}\{t\E^\nu[\W_2(\mu_t^B,\mu_0)^2|T<\si_\tau^B]\}\\
&\ge 2c_5^{-1}\e^{-c_5 r}
\sum_{m=1}^\infty\ff{\e^{-2(\ll_m-\ll_0)r}}{(\ll_m-\ll_0)[B(\ll_m)-B(\ll_0)]},\quad r>0.
\end{split}\end{equation}
Therefore, letting $r\to 0$ in \eqref{W2-r}, by Fatou's lemma, we arrive at
$$\liminf_{t\to\infty}\inf_{T\ge t}\{t\E^\nu[\W_2(\mu_t^B,\mu_0)^2|T<\si_\tau^B]\}\ge c_5^{-1}\sum_{m=1}^\infty\ff 2 {(\ll_m-\ll_0)[B(\ll_m)-B(\ll_0)]},$$
which proves \eqref{1T1.1} with $c=c_5^{-1}$.

In addition, if $\partial M$ is convex, then we can take $\kappa_1=1$ in \eqref{W-contra} by \cite[Lemma 4.6]{eW2},
and hence, we can take $c_5=1$ in \eqref{W2-r}, which immediately leads to
$$\liminf_{t\to\infty}\inf_{T\ge t}\{t\E^\nu[\W_2(\mu_t^B,\mu_0)^2|T<\si_\tau^B]\}\ge \sum_{m=1}^\infty\ff 2 {(\ll_m-\ll_0)[B(\ll_m)-B(\ll_0)]}.$$
The second assertion in Theorem \ref{T1.1}(1) is proved.

(2) Let $\alpha\in(0,1]$ and $B\in\BB_\alpha$. It suffices to prove \eqref{2T1.1} for every $p\in(0,\alpha)$, since $(\W_{q_1})^{1/q_1}\leq \W_1\leq\W_{q_2}$ for every $0<q_1\leq1\leq q_2<\infty$. Let $t\ge 1$ and $N\in \N$ such that $N\ge c t$ for some constant $c\geq1$.
Set $\bar{\mu}_N^B:=\ff 1 N \sum_{i=1}^N\delta_{X_{t_i}^B}$, where $t_i:=\ff{(i-1)t}{N}$, $1\le i\le N+1$.
Take
$$\pi_0(\d x,\d y)=\ff 1 t\sum_{i=1}^N\int_{t_i}^{t_{i+1}}\dd_{X_s^B}(\d x)\dd_{X_{t_i}^B}(\d y)\d s.$$
It is easy to see that $\pi_0\in\scr{C}(\mu_t^B,\bar{\mu}_N^B)$.
Then
\begin{equation}\label{wp-u}
\W_p(\mu_t^B,\bar{\mu}_N^B)\leq \int_{M\times M}\rho(x,y)^p\,\pi_0(\d x,\d y)=  \ff 1 t
\sum_{i=1}^N\int_{t_i}^{t_{i+1}}\rho(X_s^B,X_{t_i}^B)^p\,\d s.
\end{equation}
By the Markov property,
\begin{equation}\label{2P4.1}\E^\nu[\rho(X_{t_i}^B,X_s^B)^p1_{\{T<\si_\tau^B\}}]
=\E^\nu[1_{\{t_i<\si_\tau^B\}}P_{s-t_i}^{D,B}\{\rho(X_{t_i}^B,\cdot)^p P_{T-s}^{D,B}1\}(X_{t_i}^B)],\quad t_i<s<T.
\end{equation}

By \eqref{SDSG0}, \eqref{DPQ} and \eqref{LT}, it is immediate to obtain that
\begin{equation*}\label{sub-markov}
P_t^{D,B}1\le c_1\e^{-B(\ll_0)t},\quad t\ge 0,
\end{equation*}
for some constant $c_1>0$. From part $({\rm b})$ in the proof of \cite[Proposition 4.1]{eW2}, we see that
$$P_u^0[\rho(x,\cdot)^2\phi_0^{-1}](x)\leq c_2 u\log[1+\phi_0^{-1}(x)],\quad x\in M,\,u>0,$$
for some constant $c_2>0$. Hence, by \eqref{SDSG0}, \eqref{R} and H\"{o}lder's inequality, there exist constants $c_3,c_4>0$ such that, for every $x\in M$ and every $s\in(t_i,t_{i+1})$,
\begin{equation}\begin{split}\label{3P4.1}
&P_{s-t_i}^{D,B}[\rho(x,\cdot)^p P_{T-s}^{D,B}1](x)\\
&\le c_1\e^{-B(\ll_0)(T-s)}P_{s-t_i}^{D,B}[\rho(x,\cdot)^p](x)\\
&=c_1\e^{-B(\ll_0)(T-s)}\int_0^\infty P_u^D[\rho(x,\cdot)^p](x)\,\P_{S_{s-t_i}^B}(\d u)\\
&\le c_1\e^{-B(\ll_0)(T-s)}\int_0^\infty\big\{P_u^D[\rho(x,\cdot)^2](x)\big\}^{\ff p 2}\,\P_{S_{s-t_i}^B}(\d u)\\
&\le c_1\e^{-B(\ll_0)(T-s)}\int_0^\infty\big\{\phi_0P_u^0[\phi_0^{-1}\rho(x,\cdot)^2](x)\big\}^{\ff p 2}\,\P_{S_{s-t_i}^B}(\d u)\\
&\leq c_3 \e^{-B(\ll_0)(T-s)} \big(\phi_0\log[1+\phi_0^{-1}]\big)^{\frac{p}{2}}(x)\int_0^\infty u^{\ff p 2}\,\P_{S_{s-t_i}^B}(\d u)\\
&\le c_4\e^{-B(\ll_0)(T-s)}(s-t_i)^{\ff p {2\aa}}\big(\phi_0\log[1+\phi_0^{-1}]\big)^{\frac{p}{2}}(x),
\end{split}\end{equation}
where in the last inequality we used the fact that
$$\int_0^\infty u^{\ff p 2}\,\P_{S_{s-t_i}^B}(\d u)\leq C(s-t_i)^{\frac{p}{2\alpha}},\quad t_i<s<t_{i+1},$$
for some constant $C>0$, since $B\in\BB_\alpha$ and $0<s-t_i\le t_{i+1}-t_i=  t/N\le 1/c\leq1$ (see e.g. \cite[Section 4.2]{eWW} for a detailed calculation).

By \eqref{2P4.1} and \eqref{3P4.1}, applying \eqref{SDSG0}, Fubini's theorem, the invariance of $P_t^0$ w.r.t. $\mu_0$ and the elementary inequality $\log(1+u)\leq u$ for any $u\geq0$ , we find a constant $c_5>0$ such that, for every $s\in(t_i,t_{i+1})$,
\begin{equation}\begin{split}\label{dist-u}
&\E^\nu[\rho(X_{t_i}^B,X_s^B)^p 1_{\{T<\si_\tau^B\}}]\le c_4\e^{-B(\ll_0)(T-s)}(s-t_i)^{\ff p {2\aa}}\nu\big(P_{t_i}^{D,B}([\phi_0\log(1+\phi_0^{-1})]^{\frac{p}{2}})\big)\\
&=c_4\e^{-B(\ll_0)(T-s)}(s-t_i)^{\ff p {2\aa}} \int_0^\infty \e^{-\lambda_0 u}\mu_0\big(h\phi_0^{-1} P_u^0[\phi_0^{\frac{p}{2}-1}\log^{\frac{p}{2}}(1+\phi_0^{-1})]\big)\,\P_{S_{t_i}^B}(\d u)\\
&\le c_4\|h\phi_0^{-1}\|_\infty\mu(\phi_0)(s-t_i)^{\ff p {2\aa}}\e^{-B(\ll_0)(T-s)}\e^{-B(\lambda_0)t_i}\leq c_5 (s-t_i)^{\ff p {2\aa}}\e^{-B(\ll_0)T},
\end{split}\end{equation}
where in the last inequality we used the fact that $s-t_i\le  1/c$. Thus, by \eqref{wp-u}, \eqref{dist-u} and \eqref{SDHT-L}, there exists a constant $c_6>0$ such that
\begin{equation}\label{2P4.1+}
\E^\nu[\W_p(\mu_t^B,\bar{\mu}_N^B)|T<\si_\tau^B]\le c_6(tN^{-1})^{\ff p {2\aa}},\quad N\in\mathbb{N},\,t\geq 1.
\end{equation}

Since $M$ is compact, there exists a constant $c_7>0$ such that
$$\mu(\{y\in M: \rr(x,y)^p \le r\}) \le c_7r^{\ff d p},\quad r>0,\, x\in M,$$
which, by \cite[Proposition~4.2]{eK} (see also \cite[Corollary 12.14]{GL2000} for the Euclidean space case), implies that
\begin{equation}\label{2P4.1++}
\W_p(\bar{\mu}_N^B,\mu_0)\geq c_8 N^{-\frac{p}{d}},\quad N\in\mathbb{N},\,t\geq 1,
\end{equation}
for some constant $c_8>0$.

Therefore, \eqref{2P4.1++} and \eqref{2P4.1+} together with the triangular inequality of $\W_p$ yield that
\begin{align*}
\inf_{T\ge t}\E^\nu[\W_p (\mu_0,\mu_t^B)|T<\si_\tau^B]  &\geq \inf_{T\ge t} \E^\nu[\W_p(\mu_0,\bar{\mu}_N^B)|T<\si_\tau^B] \\
&\quad-\sup_{T\ge t} \E^\nu[\W_p(\mu_t^B,\bar{\mu}_N^B)|T<\si_\tau^B]\\
&\geq c_8 N^{-\frac{p}{d}} - c_6(tN^{-1})^{\ff p{2\aa}}, \quad t\geq 1,\, N\in \mathbb N.
\end{align*}
By taking  $N:=\inf\{n\in\mathbb N: n\ge \dd t^{\ff{d}{d-2\aa}}\}$  for some $\dd>0$, we can find a constant $c_9>0$ such that, for large enough $t>1,$
$$\inf_{T\ge t} \E^\nu[\W_p(\mu_0,\mu_t^B)|T<\si_\tau^B]  \ge c_9 t^{-\ff p{d-2\aa}},$$
which proves \eqref{2T1.1}.
\end{proof}

\subsection*{Acknowledgment}
Financial support by the National Natural Science Foundation of China (No. 11831014) is gratefully acknowledged.  The authors would like to thank Prof. Feng-Yu Wang for helpful comments.

\section*{Appendix}
In this part, we present some additional details on the proof of Lemma \ref{L3.2} for the convenience of the reader. We use the same notations as those in the proof of Lemma \ref{L3.2}.

\begin{proof}[Proof of \eqref{8L3.2}]
For each $m\in\mathbb{N}$, for every $r\in(0,1]$ and every $t>0$,  let
$$b_m:=\frac{\e^{-2(\lambda_m-\lambda_0)r}}{\lambda_m-\lambda_0},$$
and set
$${\rm E}:=t\E^\nu\big[\mu_0\big(|\nabla (-\mathcal{L}_0)^{-1}(\rho_{t,r}^B-1)|^2\big)|T<\sigma_\tau^B\big].$$
Then, by \eqref{1L3.2}, \eqref{3L3.2} with $f=\phi_m\phi_0^{-1}$ and \eqref{5L3.2}, we have
\begin{equation*}\begin{split}
{\rm E}&=\frac{2}{t\P^{\nu}(T<\sigma_\tau^B)}\sum_{m=1}^\infty b_m\int_0^t\! \int_{s_1}^t\!\int_0^\infty\!\int_0^\infty\!\int_0^\infty \e^{-\lambda_0 (u+v+w)}\big[\nu(\phi_0)\mu(\phi_0)\e^{-(\lambda_m-\lambda_0)v}\\
&\quad+({\rm J_1}+{\rm J_2}+{\rm J_3})(u,v,w)\big]\,\Pp_{S_{T-s_2}^B}(\d w)\Pp_{S_{s_2-s_1}^B}(\d v)\Pp_{S_{s_1}^B}(\d u)\d s_2 \d s_1\\
&=:\frac{2}{t\P^{\nu}(T<\sigma_\tau^B)}({\rm E}_1+{\rm E}_2),
\end{split}\end{equation*}
where
\begin{equation*}\begin{split}
{\rm E}_1&=\nu(\phi_0)\mu(\phi_0)\sum_{m=1}^\infty b_m\\
&\quad \times \int_0^t\! \int_{s_1}^t\!\int_0^\infty\!\int_0^\infty\!\int_0^\infty \e^{-\lambda_0 (u+w)}\e^{-\lambda_m v}
\,\Pp_{S_{T-s_2}^B}(\d w)\Pp_{S_{s_2-s_1}^B}(\d v)\Pp_{S_{s_1}^B}(\d u)\d s_2\d s_1,
\end{split}\end{equation*}
and
$${\rm E}_2=\sum_{m=1}^\infty b_m\int_0^t\d s_1 \int_{s_1}^t\Xi_T(s_1,s_2)\,\d s_2,$$
where $\Xi_T(s_1,s_2)$ is defined by \eqref{Xi}. It is clear, by \eqref{LT}, we  have
\begin{equation*}\begin{split}
&\int_0^t\!  \int_{s_1}^t\!\int_0^\infty\!\int_0^\infty\!\int_0^\infty \e^{-\lambda_0 (u+w)}\e^{-\lambda_m v}
\,\Pp_{S_{T-s_2}^B}(\d w)\Pp_{S_{s_2-s_1}^B}(\d v)\Pp_{S_{s_1}^B}(\d u)\d s_2\d s_1\\
&=\e^{-B(\lambda_0)T} \int_0^t\!\!  \d s_1\int_{s_1}^t \e^{[B(\lambda_m)-B(\lambda_0)]s_1}\e^{-[B(\lambda_m)-B(\lambda_0)]s_2}\,\d s_2\\
&=\e^{-B(\lambda_0)T}\Big(\frac{t}{B(\lambda_m)-B(\lambda_0)}-\frac{1-\e^{-[B(\lambda_m)-B(\lambda_0)]t}}{[B(\lambda_m)-B(\lambda_0)]^2}\Big),
\end{split}\end{equation*}
and hence
\begin{equation*}\begin{split}
{\rm E}_1&=\e^{-B(\lambda_0)T}\nu(\phi_0)\mu(\phi_0)\sum_{m=1}^\infty b_m\Big(\frac{t}{B(\lambda_m)-B(\lambda_0)}-\frac{1-\e^{-[B(\lambda_m)-B(\lambda_0)]t}}{[B(\lambda_m)-B(\lambda_0)]^2}\Big).
\end{split}\end{equation*}

Let $a_m:=B(\lambda_m)-B(\lambda_0)$, $m\in\mathbb{N}$, and
$$h_T:=\frac{\e^{-B(\lambda_0)T}\nu(\phi_0)\mu(\phi_0)}{\P^{\nu}(T<\sigma_\tau^B)},\quad T>0.$$
Then
\begin{equation*}\begin{split}
{\rm E}-2\sum_{m=1}^\infty \frac{b_m}{a_m}= \frac{2}{t\P^{\nu}(T<\sigma_\tau^B)} {\rm E}_2+ 2(h_T-1)\sum_{m=1}^\infty \frac{b_m}{a_m}
-\frac{2}{t}h_T\sum_{m=1}^\infty \frac{b_m}{a_m}\frac{1-\e^{-a_m t}}{a_m}.
\end{split}\end{equation*}
By \eqref{ST}, \eqref{SDHT-L} and \eqref{7L3.2}, there exists some constant $c_1>0$ such that
\begin{equation*}\begin{split}
|h_T-1|&=\frac{1}{\P^{\nu}(T<\sigma_\tau^B)}\left|\int_0^\infty \e^{-\lambda_0 s}\big[\mu(\phi_0)\nu(\phi_0)-
\nu(\phi_0P_s^0\phi_0^{-1})\big]\,\P_{S_T^B}(\d s)\right|\\
&\leq c_1\e^{-[B(\lambda_1)-B(\lambda_0)]T},\quad T>0,
\end{split}\end{equation*}
and
$$|h_T|\leq c_1,\quad T>0.$$
It is easy to see that
$$\frac{1-\e^{-a_m t}}{a_m}\leq c_2,\quad t>0,\,m\in\mathbb{N},$$
for some constant $c_2>0$, since $a_m\geq a_1>0$. Thus, we find a constant $c_3>0$ such that
\begin{equation*}\begin{split}
\left|{\rm E}-2\sum_{m=1}^\infty \frac{b_m}{a_m}\right|\leq  \frac{c_3}{t\P^\nu(T<\si_\tau^B)} |{\rm E}_2|
+ \frac{c_3}{t}\sum_{m=1}^\infty \frac{b_m}{a_m},\quad T\geq t>0,\,r\in(0,1],
\end{split}\end{equation*}
which implies \eqref{8L3.2} by taking the supremum on $T$ over $[t,\infty)$.
\end{proof}

\begin{proof}[Proof of \eqref{10L3.2+}]
Let $\beta:=\ff{d+2}{2\aa\theta}$. Note that $\beta\in(0,1)$ due to that $\theta\in(\frac{d+2}{2\aa},3)$
and $d<6\alpha-2$. Then there exists a constant $c_1>0$ such that
\begin{equation*}\begin{split}
&\int_{s_1}^t\e^{[B(\ll_1)-B(\ll_0)]s_2}[1+(T-s_2)^{-\beta}]\,\d s_2\le\int_{s_1}^t\e^{[B(\ll_1)-B(\ll_0)]s_2}[1+(t-s_2)^{-\beta}]\,\d s_2\\
&=\frac{\e^{[B(\ll_1)-B(\ll_0)]t}-\e^{[(B(\ll_1)-B(\ll_0)]s_1}}{B(\ll_1)-B(\ll_0)}+\int_0^{t-s_1}\e^{[B(\ll_1)-B(\ll_0)](t-u)}u^{-\beta}\,\d u\\
&\le\frac{\e^{[B(\ll_1)-B(\ll_0)]t}-\e^{[(B(\ll_1)-B(\ll_0)]s_1}}{B(\ll_1)-B(\ll_0)}
+\int_0^{\infty}\e^{[B(\ll_1)-B(\ll_0)](t-u)}u^{-\beta}\,\d u\\
&\le c_1\e^{[B(\ll_1)-B(\ll_0)]t},\quad T\geq t\ge s_1>0.
\end{split}\end{equation*}
Hence
\begin{equation}\begin{split}\label{A1}
&\int_0^t\d s_1\int_{s_1}^t \e^{-B(\ll_0)(s_2-s_1)}\e^{-B(\ll_1)s_1}\e^{-B(\ll_1)(T-s_2)}[1+(T-s_2)^{-\beta}]\,\d s_2\\
&=\e^{-B(\ll_1)T}\int_0^ t  \e^{-[B(\ll_1)-B(\ll_0)]s_1}\Big(\int_{s_1}^t\e^{[B(\ll_1)-B(\ll_0)]s_2}[1+(T-s_2)^{-\beta}]\,\d s_2\Big)\,\d s_1 \\
&\le c_2\e^{-B(\ll_1)T}\e^{[B(\ll_1)-B(\ll_0)]t},\quad T\ge t>0,
\end{split}\tag{A1}\end{equation}
for some constant $c_2>0$.

Similarly, we can find some constants $c_3,c_4>0$ such that
\begin{equation}\begin{split}\label{A2}
&\int_0^t\d s_1\int_{s_1}^t \e^{-B(\ll_0)(T-s_2)}\e^{-B(\ll_m)(s_2-s_1)}\e^{-B(\ll_1)s_1}\,\d s_2\\
&=\frac{\e^{-B(\ll_0)T}}{B(\ll_m)-B(\ll_0)}\int_0^t \e^{[B(\ll_m)-B(\ll_1)]s_1} \Big(\e^{-[B(\ll_m)-B(\ll_0)]s_1}-\e^{-[B(\ll_m)-B(\ll_0)]t}\Big)\,\d s_1\\
&\le\frac{\e^{-B(\ll_0)T}}{B(\ll_m)-B(\ll_0)} \int_0^t \e^{-[B(\ll_1)-B(\ll_0)]s_1}\,\d s_1\\
&\le c_3\e^{-B(\ll_0)T},\quad T\ge t>0,
\end{split}\tag{A2}\end{equation}
and
\begin{equation}\begin{split}\label{A3}
&\int_0^t\d s_1\int_{s_1}^t \e^{-B(\ll_0)s_1}\e^{-B(\ll_m)(s_2-s_1)}\e^{-B(\ll_1)(T-s_2)}\big[1+(T-s_2)^{-\beta}\big]\,\d s_2\\
&\leq\e^{-B(\ll_1)T}\int_0^t\int_{s_1}^t\e^{[B(\ll_1)-B(\ll_0)]s_1}\big[1+(t-s_2)^{-\beta}\big]\,\d s_2 \d s_1\\
&=\e^{-B(\ll_1)T}\int_0^t\e^{[B(\ll_1)-B(\ll_0)]s_1}(t-s_1)\,\d s_1\\
&\quad+\ff{\e^{-B(\ll_1)T}}{1-\beta}\int_0^t\e^{[B(\ll_1)-B(\ll_0)]s_1}(t-s_1)^{1-\beta}\,\d s_1\\
&\le c_4\e^{-B(\ll_1)T}\e^{[B(\ll_1)-B(\ll_0)]t},\quad T\ge t>0.
\end{split}\tag{A3}\end{equation}

Combining \eqref{9L3.2} and \eqref{10L3.2} together with \eqref{A1}, \eqref{A2} and \eqref{A3}, we have a constant $c_5>0$ such that
$$\e^{B(\ll_0) T}\int_0^t\d s_1\int_{s_1}^t\Xi_T(s_1,s_2)\,\d s_2\le c_5,\quad T\ge t>0,$$
which is \eqref{10L3.2+}.
\end{proof}

\end{document}